\documentclass[12pt,letterpaper]{amsart}
\usepackage[hypertex]{hyperref}
\usepackage{color,graphicx,array,amscd,amsmath,amssymb,MnSymbol,abbrViro}

\begin{document}
\title[Hyperfields and dequantization]{Hyperfields for Tropical 
Geometry I. Hyperfields and dequantization}
\author{Oleg Viro}
\begin{abstract} 
New hyperfields, that is fields in which addition is multivalued, are
introduced and studied. In a separate paper these hyperfields are 
shown to provide a base for the tropical geometry.
The main hyperfields considered here are classical number sets,
such as the set $\C$ of complex numbers, the set $\R$ of real numbers, 
and the set $\R_+$ of real non-negative numbers,
with the usual multiplications, but new, multivalued additions.
The new hyperfields are related with the classical fields and each other
by dequantisations. For example,
the new complex tropical field $\tc$ is a dequantization of the field 
$\C$ of complex numbers.
\end{abstract}
\maketitle

\section{Introduction}\label{s1}
\subsection{Natural, but not well-known}\label{s1.1}
This paper is devoted to hyperfields. The notion of hyperfield is an 
immediate generalization of the notion of field. 
A hyperfield is just a field, in which the addition is multivalued.
Hyperfields are very natural and useful algebraic objects. 

However, still, they have to find their way to the mainstream mathematics:
still, it is easier to re-invent them than to find out that they have been 
invented. 

When I realized that objects of this kind would be
very useful in my efforts to find an appropriate base for tropical 
geometry, it was not difficult to build up a basic theory around my
examples, but it took much longer to find the theory in literature. 

I gave several talks about this matter (in particular, a talk
\cite{Viro_msri} at MSRI workshop {\em Tropical Structures in Geometry 
and Physics\/} on November 30, 2009). The
talks were attended by many mathematicians, but nobody told me about 
acquaintance with generalized fields having a multivalued addition. 
I am very grateful to Anatoly Vershik: when he listened to my talk in mid
January, 2010, he told me that there were many papers devoted to multigroups
and the likes. I found by Google a paper \cite{Marshall} of 2006 
by Murray Marshall, where multiring and multifields were defined.
 
In the introduction to \cite{Marshall} Marshall wrote: "The idea of a 
multiring is very natural, although there seems to be no reference to it 
in the literature. Some basic properties of multigroups and multirings are
established in Sections 1 and 2." In Section 2 Marshall defined also a 
multifield. 

I have changed a draft of my paper, replacing my term "tropical field" 
by Marshall's "multifield" and uploaded it to arXiv as \cite{Viro_mf}. 
Of course, I contacted Marshall and expressed to him my excitement about 
multifields. 

A couple of months later Marshall informed me that he learned from recent
preprints \cite{CC1}, \cite{CC2} of Alain Connes and Caterina Consani 
that our multifields under the name of hyperfileds were introduced as early
as in 1956 by Marc Krasner  \cite{Krasner1}. Krasner \cite{Krasner1}
introduced also hyperrings, but the notion of multiring
introduced by Marshall \cite{Marshall} is more general than the notion
of hyperring considered by Krasner, and the difference is essential for 
the applications that Marshall developed. 
Therefore both terms will be used. But Krasner's hyperfield and  
Marshall's multifield are absolutely the same, and the term hyperfield 
wins as it is older. 

This paper is a new version of \cite{Viro_mf}. I insert the corresponding
correction of references and terminology and make a few new remarks inspired
by new information that I got from the papers \cite{CC1}, \cite{CC2} of 
Alain Connes and Caterina Consani. 

Krasner, Marshall, Connes and Consani and the author came to hyperfields
for different reasons, motivated by different mathematical problems, 
but we came to the same 
conclusion: the hyperrings and hyperfields are great, very useful
and very underdeveloped in the mathematical literature.

Probably, the main obstacle for hyperfields to become a mainstream notion 
is that a multivalued operation does not fit to the tradition of
set-theoretic terminology, which forces to avoid multivalued maps 
at any cost. 

I believe the taboo on multivalued maps has no real ground, and eventually 
will be removed. Hyperfields, as well as multigroups, hyperrings and 
multirings, are   
legitimate algebraic objects related in many ways to the classical core of
mathematics. They provide elegant terminological and conceptual 
opportunities. In this paper I try to present new evidences for this.

I rediscovered hyperfields in an attempt \cite{Viro_msri} to find a true  
algebraic background of the tropical geometry. I believe  
hyperfields are to displace the tropical semifield in the tropical
geometry. They suit the role better.
In particular, with hyperfields the varieties are defined by equations, 
as in other branches of algebraic geometry. 

\subsection{Results}\label{s1.2}
The main results of this paper are 
new examples of hyperfields. The underlying sets of these hyperfields 
are classical: the set $\C$ of all complex numbers, the set $\R$ of all 
real numbers, the set $\R_+$ of non-negative real numbers. Multiplication
also is the usual multiplication of numbers. The additions are new
multivalued operations. 

These hyperfields are related to each other and the classical fields 
by hyperfield homomorphisms, and also via  
degenerations of the structures, similar to the Litvinov-Maslov
dequantization \cite{L-M}, which relates the semifield $(\R_+,+,\times)$
of non-negative real numbers with the usual arithmetic operations to 
the tropical semifield $\T=(\R\cup\{0\},\max,+)$. 
In particular, the fields $\C$ and $\R$ are dequantized. 
I call the results complex tropical hyperfield $\tc$ and real tropical
hyperfield $\tr$.
 
A new hyperfield that does not appear via dequantizing a field, is a
{\sfit triangle hyperfield\/} $\mftr$. Its underlying set is $\R_+$, and
the addition is related to the triangle inequality: the sum of two
non-negative numbers $a$ and $b$ is defined as the set of non-negative
numbers $c$ such that there exists an Euclidean triangle with sides of
lengths $a$, $b$ and $c$. This hyperfield dequantizes to a similar
hyperfield $\mfutr$ in which addition is related in the same way with 
the ultra-metric triangle inequality $c\le\max(a,b)$.   

Applications of the hyperfields introduced in this paper 
to the tropical geometry will be presented in a separate papers
\cite{Viro_eqmfd} and 
\cite{Viro_tg}, a preliminary exposition can be found in \cite{Viro_bntg}.

\subsection{Organization of the paper}\label{s1.3}
Section \ref{s2} is devoted to the general multivalued algebra. 
It starts with a discussion of the terminology related to 
multivalued maps. Then multivalued binary operations are discussed. 

In Section \ref{s3}, the notions related to 
multivalued generalizations of groups are discussed. This
discussion is not complete, due to long history and a huge number of
various level of the generalizations. We concentrate mainly on the notions
needed to what follows.

In Section \ref{s4} we turn to multirings, hyperrings and hyperfields, their
examples and general properties. Section \ref{s4} finishes with a 
discussion of multiring homomorphisms, their examples and first
applications.

In Section \ref{s5} a few 
hyperfields related to triangle inequalities are introduced
(triangle, ultra-triangle, tropical and amoeba hyperfields).

In Section \ref{s6} we introduce tropical addition of complex numbers 
and discuss its properties. In Section \ref{s7} subhyperfields of the 
complex tropical hyperfield are considered.

In Section \ref{s9} the dequantization are considered. We start with
the Litvinov-Maslov dequantization, then study dequantization 
of the triangle hyperfield to the ultratriangle one, and 
dequantization of the field $\C$ to the complex tropical hyperfield.  
All the dequantizations are related to each other at the end of Section
\ref{s9}.

\subsection{Acknowledgemnts}\label{s1.4}
The research presented in this paper was partly made when the author 
participated in a semester program on tropical geometry at MSRI. I am
grateful to MSRI for an excellent environment and opportunity of direct 
contact with the leading specialists in tropical geometry. I am grateful to 
G.~Mikhalkin, Ya.~Eliashberg, V.~Kharlamov, I.~Itenberg, L.~Katsarkov, 
I.~Zharkov, A.~M.~Vershik, M.~Marshall, A.~Connes and C.~Consani 
for useful discussions. 

\tableofcontents

\section{Multivalued maps and operations}\label{s2}

Multivalued operations hardly belong to the mainstream of  conventional 
mathematics, but they appear here and there. In this section the basic 
terminology related to multivalued maps is introduced and discussed. 

\subsection{Multivalued mappings}\label{s2.1}
For a set $X$, the symbol $2^X$ denotes the set of all subsets of $X$.
A {\sfit multivalued map\/} or {\sfit multimap\/} of a set $X$ to a set $Y$ 
is a map 
$X\to2^Y$, which is treated for some reasons as a map $X\to Y$ 
that does not satisfy the usual requirement of being univalent
(according to this requirement a map must take each element of $X$ 
to a {\sfit single\/} element of $Y$). 

The reason for considering a multivalued map is usually a desire to 
emphasize an analogy to another
situation, in which the corresponding map is univalued.  In this paper we study
a generalization of addition with the sum allowed to be multivalued. 
Usage of the modern set-theoretic terminology would make analogies with the
usual addition more difficult to recognize. Cf., for example,
\cite{Marshall}, where a multivalued binary operation is introduced,
according to the standards of set theory, as a
subset of the Cartesian cube of the underlying set, 
but a couple of pages after that the multivalued notation take over, anyway. 
Therefore we dare to use less conventional terminology of multivalued maps.   

The term {\sfit set-valued\/} is used as a
synonym for multivalued. A multivalued map $f$ of $X$ to $Y$ is denoted
by  $f:X\multimap Y$. 

\subsection{Adjustment of terminology}\label{s2.2}
As in other cases of disrespect towards the standards of set-theoretic
terminology, this one implies a whole chain of modifications of commonly
accepted terminology and  notation. 
Some of the modifications are straightforward 
and cannot lead to a confusion. For example, the {\sfit value\/} $f(a)$ 
at $a\in X$ 
is the subset of $Y$ which is the image of $a$ under the corresponding map 
$X\to 2^Y$. What happens to the notion of the {\sfit image\/} of a set is less 
logical, but still easy to guess: 
for $A\subset X$, the symbol $f(A)$ denotes not the subset
$\{f(x)\mid x\in A\}$ of $2^Y$, but the subset $\cup_{x\in A}f(x)$ of  $Y$.

In the same spirit: the {\sfit composition\/} of multivalued 
maps $f:X\multimap Y$ 
and $g:Y\multimap Z$ is a multimap $g\circ f:X\multimap Z$ 
that takes  $a\in X$ to $g(f(a))=\cup_{y\in f(a)}g(y)$. 

Other modifications may be quite confusing. For example, what is the
preimage of a set $B\subset Y$ under a multivalued map $f:X\multimap Y$?
The set $\{a\in X\mid f(a)\subset B\}$ or $\{a\in X\mid f(a)\cap
B\ne\varnothing\}$? We see that the notion of the preimage of a set splits
under the transition from univalued maps to multivalued ones. In cases of such
ambiguity one needs to adjust the terminology. For example, 
the set   $\{a\in X\mid f(a)\subset B\}$ is called the {\sfit upper
preimage\/} of $B$ under $f$ and denoted by $f^+(B)$, while
$\{a\in X\mid f(a)\cap B\ne\varnothing\}$ is called the {\sfit lower
preimage\/} of $B$ under $f$ and denoted by $f^-(B)$. The names seems to 
be confusing because $f^+(B)\subset f^-(B)$, the upper preimage is smaller  
than the lower one.  
 
In order to take refuge in the 
standard set-theoretic terminology, we will pass from a multivalued map 
$f:X\multimap Y$ to the corresponding univalued map $X\to2^Y$. The latter
will be denoted by $f^\uparrow$.

Sets, multimaps and their compositions form a category. Thus, although
multivalued maps do not quite comply with the set-theoretic terminology, 
they fit comfortably to a more modern category-theoretic setup.    

\subsection{Multivalued binary operations}\label{s2.3} 
A multivalued map $X\times X\multimap X$ with non-empty values 
is called a {\sfit binary multivalued operation\/} in $X$.

A binary multivalued operation $f:X\times X\multimap X$ is said to be  
{\sfit commutative\/} if $f(a,b)=f(b,a)$ for any $a,b\in X$.

A binary multivalued operation $f:X\times X\to 2^X$ is said to be  
{\sfit associative\/} if $f(f(a,b),c)=f(a,f(b,c)$ for any $a,b,c\in X$.
Certainly, in the latter formula, by $f$ we mean not only $f$, but also 
its natural extension to all subsets of  $X$, that is 
$$2^X\times 2^X\to2^X:(A,B)\mapsto\bigcup_{a\in A,b\in B}f(a,b).$$

Let $Y\subset X$ and $f:X\times X\multimap X$ be a multivalued binary
operation. A multivalued binary operation  $g:Y\times Y\multimap Y$ is said
to be {\sfit induced \/} by $f$, if $g(a,b)=f(a,b)\cap Y$ for any  
$a,b\in Y$. Of course, the induced operation is completely determined by
the original one. It exists iff  $f(a,b)\cap Y\ne\varnothing$ 
for any  $a,b\in Y$ (recall that according to the definition of a multivalued 
operation the set  $g(a,b)$ is not allowed to be empty). 

\section{Hypergroup-Multigroups-Polygroups}\label{s3}

\subsection{Definition of multigroup}\label{s3.1}
A set $X$ with a {\it multivalued\/} binary operation 
$(a,b)\mapsto a\cdot b$ 
is called a {\sfit multigroup} if 
\begin{enumerate} 
\item the operation $(a,b)\mapsto a\cdot b$ is associative;
\item $X$ contains an element $1$ such that $1\cdot a=a=a\cdot 1$ for 
any $a\in X$;
\item for each $a\in X$ there exists a unique $b\in X$ such that
$1\in a\cdot b$ and there exists a unique $c\in X$ such that 
$1\in c\cdot a$. Furthermore, $b=c$. This element is denoted by $a^{-1}$. 
\item $c\in a\cdot b$ iff $c^{-1}\in b^{-1}\cdot a^{-1}$ for any $a,b,c\in X$.
\end{enumerate}
This is a straightforward generalization of the notion of group:
any group is a multigroup and 
a multigroup, in which the group operation is univalued 
(i.e., $a\cdot b $ consists of a single element) is a group.

The axioms of multigroup presented above are not minimal. I have chosen 
them, because they give a good idea what is the structure and why
multigroups generalize groups. To my taste, they are convenient if you
know already that you deal with a multigroup and want to deduce something
from axioms.
For checking if something is a multigroup, a more minimalistic set of 
axioms, like the one provided by following theorem,  may serve better.  

\begin{Th}[Cf. Marshall \cite{Marshall}]\label{th-reform-multigroup}
A set $X$ with a multivalued operation $(a,b)\mapsto a\cdot b$ is a
multigroup iff there exist a map $X\to X:a\mapsto a^{-1}$ and an element
$1\in X$ such that:
\begin{enumerate} 
\item[(1')] $(a\cdot b)\cdot c\subset a\cdot(b\cdot c)$ for any $a,b,c\in X$,
\item[(2')] $a\cdot 1=a$ for any $a\in X$,
\item[(3')] {\sfit Reversibility property.} $c\in a\cdot b$ implies 
$a\in c\cdot b^{-1}$ and $b\in a^{-1}\cdot c$ for $a,b,c\in X$.
\end{enumerate}
\end{Th}

\begin{proof} First, let us prove that statements (1') - (3') hold true in
any multigroup. Obviously, (1') follows from axiom (1), and (2') follows
from axiom (2). Let us prove (3').

If  $c\in a\cdot b$, then $1\in (a\cdot b)\cdot c^{-1}=a\cdot(b\cdot
c^{-1})$. By axiom (3), $a^{-1}$ is the unique element
$x$ such that $1\in ax$. Therefore $a^{-1}\in b\cdot c^{-1}$. By axiom (4), 
$a^{-1}\in b\cdot c^{-1}$ iff $a\in c\cdot b^{-1}$.
Thus, we proved that  $c\in a\cdot b$ implies $a\in c\cdot b^{-1}$. The
other implication is proved similarly: $c \in a\cdot b$ implies 
$1\in c^{-1}\cdot(a\cdot b)=(c^{-1}\cdot a)\cdot b$,
therefore $b^{-1}\in c^{-1}\cdot a$. Hence $b\in a^{-1}\cdot c$.

Now let us deduce the axioms (1) - (4) from (1') - (3'). 

First, observe that $1^{-1}=1$. Indeed, $1\cdot 1=1$ by (2'), 
hence $1\in 1^{-1}\cdot1$ by (3'), and $1^{-1}\cdot1=1^{-1}$ by (2').

Second, observe that $1\in a^{-1}\cdot a$. Indeed, $a\in a\cdot1$ by (2'),
and by applying (3') we get $1\in a^{-1}\cdot a$.

Now let us prove that the map $X\to X:a\mapsto a^{-1}$ is an
involution, that is $(a^{-1})^{-1}=a$ for any $a\in X$. We have just proved
that $1\in a^{-1}\cdot a$. By (3'), it follows $a\in(a^{-1})^{-1}\cdot1$,
but by (2') $(a^{-1})^{-1}\cdot1=(a^{-1})^{-1}$.    

Now let us deduce axiom (4). Assume that $c\in a\cdot b$. By (3') this
implies that  $a\in c\cdot b^{-1}$. Again by (3'), this implies
$b^{-1}\in c^{-1}\cdot a$. Finally, again by (3'), this implies
$c^{-1}\in b^{-1}\cdot a^{-1}$. The opposite implication 
$c^{-1}\in b^{-1}\cdot a^{-1}\implies c\in a\cdot b$ follows from the one
that we proved by substituting $a^{-1}$ for $a$, $b^{-1}$ for $b$ and
$c^{-1}$ for $c$ and using the fact the $a\mapsto a^{-1}$ is an involution.

Now let us deduce (2). One of the two equalities constituting (2) is (2').
The other one follows from (2'), (4) and the fact that $a\mapsto a^{-1}$ is
a bijection (as an involution).

In order to deduce (3), we need to prove that from (1') - (3') it follows
that $1\in a\cdot b$ implies $b=a^{-1}$ and $a=b^{-1}$. Apply (3') to
$1\in a\cdot b$. It gives $a\in 1\cdot b^{-1}$ and $b\in a^{-1}\cdot 1$.
By (2) this implies  $b=a^{-1}$ and $a=b^{-1}$.
 
In order to prove (1), we need to prove the inclusion  
$(a\cdot b)\cdot c\supset a\cdot(b\cdot c)$. By (1'), 
$(c^{-1}\cdot b^{-1})\cdot a^{-1}\subset c^{-1}\cdot(b^{-1}\cdot a^{-1})$.
Applying the involution $x\mapsto x^{-1}$ to both sides of this
inclusion, we obtain
$\left((c^{-1}\cdot b^{-1})\cdot a^{-1}\right)^{-1}\subset 
\left(c^{-1}\cdot(b^{-1}\cdot a^{-1})\right)^{-1}$. 
Then, applying axiom (4), which has already been deduced above, we obtain  
$a\cdot (b\cdot c)\subset (a\cdot b)\cdot c$.
\end{proof} 

Notice, that in the proof of Theorem \ref{th-reform-multigroup}
we proved that in any multigroup  $1^{-1}=1$ and $(a^{-1})^{-1}=a$.

\subsection{Notation}\label{s3.2}
In what follows we meet mostly {\sfit commutative\/} multigroups.
Then we will use an additive notation: the neutral element $1$ will be 
denoted by $0$, the element $a^{-1}$ will be denoted by $-a$, 
the multigroup operation will be denoted by various symbols such as 
$\tplus$, $\cplus$, $\yplus$, $\nplus$.
We use these symbols (instead of commonly used $+$), because the 
multivalued operations  
will be considered below in an environment where
the usual addition $(a,b)\mapsto a+b$ is also present, and, moreover, two 
multivalued additions may be considered simultaneously.

\subsection{The smallest multigroup}\label{s3.3} The  smallest
multigroup which is not a group: in the set $\{0,1\}$ 
define an operation $\yplus$  by formulas:
$0\yplus 0=0$, $0\yplus 1=1=1\yplus 0$, $1\yplus 1=\{0,1\}$.
One can easily check that this is a multigroup.
Following Marshall \cite{Marshall}, we denote this multigroup by $Q_1$. 
This is the only multigroup of two elements that is not a group. 

\subsection{Multigroups of a linear order}\label{s3.4}
$Q_1$ belongs to a family of multigroups defined by linearly ordered sets.
Let $X$ be a linearly ordered set with order $\prec$ and an element $0$
such that $0\prec x$ for any $x\in X$ different from $0$. 
Define in $X$ a binary multivalued operation 
$$
(a,b)\mapsto a\yplus b=
\begin{cases} 
\max(a,b), &\text{ if } a\ne b\\
\{x\in X\mid x\preceq a\}, &\text{ if }a=b.
\end{cases}
$$
It is easy to verify that $X$ with $\yplus$ is a multigroup and 
$-a=a$ for any $a\in X$.

 This construction gives $Q_1$ if $X=\{0,1\}$ and $0\prec1$.

In the same situation $X$ can be turned into a different multigroup.
For this, define a binary multivalued operation
$$
(a,b)\mapsto a\upspoon b=
\begin{cases} 
\max(a,b), &\text{ if } a\ne b\\
\{x\in X\mid x\prec a\}, &\text{ if }a=b\ne0\\
0, &\text{ if }a=b=0.
\end{cases}
$$
It is easy to verify that $X$ with $\upspoon$ is a multigroup and 
$-a=a$ for any $a\in X$. For $X=\{0,1\}$ and $0\prec1$ this construction
gives a group. If $X$ consists of more than 2 elements, the operation
$\upspoon$ is truly multivalued. 

We will call $(X,\yplus)$ a {\sfit linear order
multigroup,\/} and $(X,\upspoon)$ a {\sfit strict linear order 
multigroup.\/}

\subsection{Three element multigroups}\label{s3.5}
Define in a three element set $\{-1,0,1\}$ operation $\cplus$ by formulas
$0\cplus x=x\cplus0=x$ and $x\cplus x=x$ for any $x$, and
$-1\cplus1=1\cplus(-1)=\{-1,0,1\}$. One can easily check that this is a
multigroup. Following Marshall \cite{Marshall}, 
we denote this multigroup by $Q_2$.

Yet another multigroup of three elements can be defined as follows.
In $\{0,1,2\}$ define operation $\tplus$ by formulas 
$0\tplus x=x\tplus 0=x$ for any $x$, $1\tplus 1=2$,
$1\tplus2=2\tplus1=\{0,1\}$, $2\tplus2=\{1,2\}$. Denote this multigroup by
$M$.   

\subsection{Multigroups of double cosets}\label{s3.6}
Traditional examples of multigroups come from the group theory. 
Let $G$ be a group, and $H$ be a subgroup of $G$. Let $X$ be the set of
double cosets, $X=\{HgH\mid g\in G\}$. Define a binary multivalued
operation $(HaH)\cdot(HbH)=\{HahbH\mid h\in H\}$. This is a multigroup, see
Dresher and Ore \cite{DO}.

\subsection{Multigroup homomorphisms}\label{s3.7}
Let $X$ and $Y$ be multigroups. A map $f:X\to Y$ is called a {\sfit 
(multigroup) homomorphism\/} if $f(e)=e$ and 
$f(a\cdot b)\subset f(a)\cdot f(b)$  for any $a,b\in X$. 

A multigroup homomorphism $f:X\to Y$ is said to be  {\sfit strong\/} 
if $f(a\cdot b)= f(a)\cdot f(b)$ for any $a,b\in X$. If $Y$ is a group,
then any multigroup homomorphism $f:X\to Y$ is strong. \smallskip

\noindent{\bfit Example.\/} 
If $X$ and $Y$ are linearly ordered sets with the smallest 
elements $0_X$ and 
$0_Y$, respectively, then any monotone map $X\to Y$ mapping $0_X$ to 
$0_Y$ is a multigroup homomorphism. Such a map is a strong homomorphism 
iff it is injective on the complement of the preimage of $0_Y$.

\subsection{Submultigroups}\label{s3.8}
Let $X$ be a multigroup with neutral element $e$, and $Y\subset X$. If
$e\in Y$, the multigroup operation in $X$ induces a binary multivalued 
operation in $Y$ and together with any $a\in Y$ the inverse element $a^{-1}$
is contained in $Y$, then $Y$ with the induced operation is a multigroup.
It is called a {\sfit submultigroup\/} of $X$. 
The inclusion $Y\hookrightarrow X$ is a multigroup homomorphism. 
If this is a strong homomorphism, then $Y$ is said to be a {\sfit strong
submultigroup\/} of $X$.

A strong submultigroup $Y$ of a multigroup $X$ is said to be {\sfit
normal,\/} if $a^{-1}\cdot Y\cdot a=Y$ for any $a\in X$. Observe, that a
normal submultigroup of $X$ contains the set $a\cdot a^{-1}$ for any 
$a\in X$.

For a multigroup homomorphism $f:X\to Y$, the set $\{a\in X\mid f(a)=e\}$
is called the {\sfit kernel\/} of $f$ and denoted by $\Ker f$. 
Obviously, this is a normal submultigroup of $X$. 

\subsection{Factorization of a multigroup homomorphism}\label{s3.9}
As in the group theory, for any normal submultigroup $Y$ of a multigroup $X$
one can construct the quotient $X/Y$, and a multigroup structure in $X/Y$ such
that the projection $X\to X/Y$ is a strong multigroup homomorphism. Any
multigroup homomorphism $f:X\to Y$ admits a natural factorization 
$X\to X/\Ker f\to Y$. 

\begin{Th}\label{th-strong-hom}
If $f$ is surjective and strong, then the induced
 multigroup homomorphism $\bar f:X/\Ker f\to Y$ is an isomorphism. 
\end{Th}

\begin{proof}
Let $\Ga,\Gb\in X/\Ker f$ and their images $\bar f(\Ga)$, $\bar f(\Gb)$ 
under $\bar f:X/\Ker f\to Y$ coincide. 
Take representative $a,b\in X$ of $\Ga$ and $\Gb$, respectively. Then 
$f(a)=\bar f(\Ga)=\bar f(\Gb)=f(b)$. 
Since $f$ is strong, $f(b^{-1}a)=f(b)^{-1}f(a)=f(a)^{-1}f(a)\ni1$.
Thus $b^{-1}a\cap\Ker f\ne\varnothing$.  Therefore there exists 
$c\in b^{-1}a\cap\Ker f$. Then $a\in bc\subset a\Ker f$ and $\Ga=\Gb$. 
\end{proof}
 
 The assumption that $f$ is strong is necessary here.  
Without this assumption, a multigroup homomorphism with a trivial kernel 
may be non injective.
 On the other hand, most of interesting multigroup homomorphisms 
are not strong. This is a major new phenomenon distinguishing multigroups
from groups.
 
Here is the simplest example: $Q_2\to Q_1:1,-1\mapsto1,0\mapsto0$.  
It is easy to see that $f$ is a
multigroup homomorphism with $\Ker f=\{0\}$, but $f$ is not injective.
In order to verify that $f$ is not strong, consider $f(1\cdot 1)=f(1)=0$,
on the other hand, $f(1)\cdot f(1)=1\cdot 1=\{0,1\}$.

\subsection{Remarks on the history of multigroups}\label{s3.10}
The notion of multigroup appeared in the
literature in various contexts, sometimes under other names (such as 
{\sfit hypergroup\/} and {\sfit polygroup\/}). 
The earliest papers \cite{Marty}, \cite{Wall1} about them 
that I could find are dated by 1934.
Some of the authors who introduced these
objects apparently were not aware on their predecessors. 
 
Often the terms multigroup and hypergroup were used for objects of wider 
classes. 
For example, Dresher and Ore \cite{DO} used the word
multigroup for much wider class of object, while what is called 
multigroup above, Dresher and Ore \cite{DO} would call a regular 
reversible in itself multigroup with an absolute unit. 

The definition given in Section \ref{s3.1} seems to be the narrowest 
and closest multivalued
generalization of the notion of group. In comparatively recent literature
exactly the same notion was considered by S.~D.~Comer \cite{Comer} (under
the name of {\sfit polygroup\/}) and M.~Marshall \cite{Marshall}. 
A.~Connes and C.~Consani \cite{CC2} consider the same notion under 
the name of {\sfit hypergroup\/}, but restrict themselves to 
commutative hypergroups. 

There is another breed of multigroups in which the value of the operation 
contains a fixed number of elements some of which may coincide to each
other. Thus the operation takes values in the $n$th symmetric power of the
set rather than in the set of all its subsets. This kind of multigroups 
was considered by Wall \cite{Wall2} and, more recently, by Buchstaber and 
Rees \cite{BR}. The author is not aware about any construction which would
allow to relate multigroups of this kind with multigroups defined in
Section \ref{s3.1}.  

\section{Multirings, hyperrings and hyperfields}\label{s4}

\subsection{Multirings}\label{s4.1}
A set $X$ equipped with a binary multivalued operation  $\tplus$ and  
(univalued) multiplication $\cdot$ is called a  {\sfit multiring\/} if 
\begin{itemize}
\item $(X,\tplus)$ is a commutative multigroup, 
\item $(X,\cdot)$ is a monoid with unity $1$ (i.e., multiplication 
$(a,b)\mapsto a\cdot b$ is associative and $1\cdot a=a=a\cdot 1$ 
for any $a\in X$),
\item $0\cdot a=0$ for any $a\in X$. 
\item the multiplication is distributive over  $\tplus$
in the sense that for every $a\in X$ maps $X\to X$ defined by formulas
$x\mapsto a\cdot x$ and $x\mapsto x\cdot a$ are homomorphisms of multigroup
$(X,\tplus)$ to itself.
\end{itemize}
 
The distributivity means that $a\cdot(b\tplus c)\subset a\cdot b\tplus
a\cdot c$ and
$(b\tplus c)\cdot a\subset (b\cdot a)\tplus(c\cdot a)$ 
for any $a,b,c\in X$. 

A multiring is said to be {\sfit commutative\/} if the multiplication is
commutative.

\subsection{Hyperrings}\label{s4.1+}
If in a multiring $X$ distributivity holds in a stronger form: 
$a\cdot(b\tplus c)= a\cdot b\tplus
a\cdot c$ and
$(b\tplus c)\cdot a= (b\cdot a)\tplus(c\cdot a)$ 
for any $a,b,c\in X$, then $X$ is called a {\sfit hyperring.\/}

An equivalent description of this strong form of distributivity:
{\it for every $a\in X$ maps $X\to X$ defined by formulas
$x\mapsto a\cdot x$ and $x\mapsto x\cdot a$ are {\sfit strong} 
homomorphisms of the multigroup $(X,\tplus)$ to itself.}

Hyperrings were introduced by Krasner \cite{Krasner1}, see also
\cite{Krasner2} and \cite{CC1}. Multirings were introduced by Marshall
\cite{Marshall}. In Marshall's work the extra generality of the notion of
multirings is used:
some of the multirings that he considers in \cite{Marshall} are not 
hyperrings.  

\subsection{Hyperfields}\label{s4.2}
A multiring  $X$ is called  a {\sfit hyperfield\/} if
$X\sminus 0$ is a commutative group under multiplication. 

\begin{Th}[See Marshall \cite{Marshall}]\label{Th:hyperfield-hyperring}
Any hyperfield is a hyperring.
\end{Th}
\begin{proof} This theorem claims that 
 in a hyperfield distributivity holds in a strong form: 
$a(b\tplus c)= (ab)\tplus(ac)$. Indeed, the
inclusion $a(b\tplus c)\subset (ab)\tplus(ac)$
that holds true in  multirings 
implies the opposite inclusion if $a\ne0$:
$$(ab)\tplus (ac)=a^{-1}a((ab)\tplus(ac))\subset
a(a^{-1}ab\tplus a^{-1}ac)=a(b\tplus b).$$
For $a=0$, the equality $a(b\tplus c)= (ab)\tplus(ac)$ holds true since
both sides are equal to $0$.
\end{proof}

The notion of hyperfield is a direct generalization of the notion of field:
a field is a hyperfield, in which the addition is univalued.

\subsection{Double distributivity}\label{s4.3}
A multivalued addition creates various new phenomena, some of which 
may be quite unexpected. 

For example, in a usual ring, distributivity 
implies that $(a+b)(x+y)=ax+ay+bx+by$. In a multiring and even in
a hyperfield the proof fails. Moreover, the equality 
$$(a\tplus b)(x\tplus y)=ax\tplus ay\tplus bx\tplus by$$ 
may be incorrect, see Sections \ref{s5.1}, \ref{s6.4}.

Let us analyze, why
the arguments that deduce  $(a+b)(x+y)=ax+ay+bx+by$ from distributivity
for univalued addition do not work for multivalued addition.
In the univalued case, $x+y$ is just an element, and one can apply
distributivity: $(a+b)(x+y)=a(x+y)+b(x+y)$. Then for each summand
distributivity is applied again, giving the equality. 

In the case of multivalued addition $\tplus$, $(x\tplus y)$ is not an
element, but a set. Therefore the distributivity $(a\tplus b)c=ac\tplus
bc$, in which $c$ is a single element (that is an axiom in a multiring) 
cannot be applied in the situation when $c$ is a set $x\tplus y$.

\begin{Th}\label{th-half-double-distr}
In any multiring,
$(a\tplus b)(x\tplus y)\subset ax\tplus ay\tplus bx\tplus by$.
\end{Th}

\begin{proof}
For each element $c\in(x\tplus y)$ the distributivity
gives $(a\tplus b)c=ac\tplus bc$, and we get 
$(a\tplus b)(x\tplus y)=
\bigcup_{c\in(x\tplus y)}(a\tplus b)c=
\bigcup_{c\in(x\tplus y)}(ac\tplus bc)$.
On the other hand,
\begin{multline*}
ax\tplus ay\tplus bx\tplus by=
 (ax\tplus ay)\tplus(bx\tplus by)=\\
 a(x\tplus y)\tplus b(x\tplus y)\supset
 ac\tplus bc
\end{multline*}
for any $c\in(x\tplus y)$, and therefore 
$$ax\tplus bx\tplus ay\tplus by\supset 
\bigcup_{c\in(x\tplus y)}(ac\tplus bc)=(a\tplus b)(x\tplus y).
$$
\end{proof}

The opposite inclusion 
$(a\tplus b)(x\tplus y)\supset ax\tplus ay\tplus bx\tplus by$ in some 
multirings does not hold true (see Section \ref{s6.4}). However, there are
multirings in which it is true. Such multirings will be called {\sfit
doubly distributive.\/} 

In a doubly distributive multiring,
$\left({\displaystyle\top}_{i=1}^na_i\right)
\left({\displaystyle\top}_{j=1}^mb_j\right)=
{\displaystyle\top}_{i,j}a_ib_j
$.
This can be easily proved by induction over $m$ and $n$.

\subsection{The smallest hyperfields}\label{s4.4}
Multigroups $Q_1$ and $Q_2$ defined in Sections \ref{s3.3} and \ref{s3.5} 
above turn
into hyperfields in a unique way.
In $Q_1=\{0,1\}$ the multiplicative group is trivial and all the products 
are defined by axioms. 
 In $Q_2=\{-1,0,1\}$ the
multiplicative group is of order 2, therefore it is uniquely defined up to
isomorphism.  

Following  Connes and Consani \cite{CC1}, we will call the 
two-element hyperfield $Q_1$ the {\sfit Krasner hyperfield\/} and denote it 
by $\K$. It can be obtained from any field $k$ with more
than two elements by identifying all invertible elements. This is a 
multiplicative factorization (see Section \ref{sn4.11} below) that was 
invented by Krasner \cite{Krasner2}.
To the best of my knowledge, $\K$ did not appear in Krasner's
papers. 

The hyperfield $Q_2$ is called the {\sfit sign hyperfield\/} and denoted by
$\S$. 

These two hyperfields are doubly distributive. 

Multigroup $M$ defined also in Section \ref{s3.5} above cannot 
be turned into a
hyperfield, unless a multivalued multiplication would be allowed.
In this paper I prefer to stay with univalued multiplications only.
If a multiplication in a hyperfield was allowed to be multivalued, 
one could define $1\cdot x=x$ and $0\cdot x=0$ for any $x$ and
$2\cdot2=\{1,2\}$. Then the multiplicative multigroup of $M$ would be 
isomorphic to $Q_1$. 

\subsection{Characteristics}\label{s4.4+}
The notion of characteristic of ring splits when we pass to multirings.

Recall that the characteristic of a ring is the smallest positive integer
$n$ such that the sum $1+\dots+1$ of $n$ summands is 0, and zero if any 
sum $1+\dots+1$ does not vanish. 

This definition can be reformulated as follows: an integer $n$
is the characteristic of a ring if $n$ is the smallest positive number 
such that the sum $1+\dots+1$ of $n+1$ summands equals 1; 
if  $1+\dots+1\ne1$ for any number $k>1$ of summands, 
than the characteristic is zero.

For multirings, straightforward generalizations of these two definitions 
are not equivalent. I propose to preserve
the old term of characteristic for the number defined by a generalization 
of the first definition, and to call the second
one {\sfit C-characteristic\/} in honor of A.~Connes and C.~Consani, who
discovered the opportunity
of speaking about hyperrings of characteristic one, 
see \cite{CC1} and \cite{CC2}.

A natural number $n$ is called the {\sfit characteristic\/} of a multiring
if this is the smallest natural number such that
$0\in1\tplus1\tplus\dots\tplus1$ where the number of summands on the
right hand side is $n$. A multiring which has no finite characteristic 
$n\ge 2$ is said to be of characteristic 0. The characteristic of a
multiring $X$ is denoted by $\chr X$.

A natural number $n$ is called the {\sfit C-characteristic\/} of a
multiring if $n$ is the smallest positive number 
such that the sum $1\tplus\dots\tplus1$ of $n+1$ summands 
{\sfit contains\/} 1; 
if  $1\not\in1\tplus\dots\tplus1$ for any number $k>1$ of summands, 
than the C-characteristic is zero.  The C-characteristic of a
multiring $X$ is denoted by $\cchr X$.

Obviously, a multiring of characteristic $p\ne0$ has C-characteristic 
$\le p$. On the other hand,  
$\chr\S=0$ and $\cchr\S=1$, while $\chr\K=2$ and $\cchr\K=1$.

A multiring is said to be {\sfit idempotent,\/} if $a\tplus a=a$ for
any $a$ in it. A multiring is idempotent iff $1\tplus 1=1$ in it.

An idempotent multiring has C-characteristic 1, but the converse is not true:
in a multiring of C-characteristic 1 the set $1\tplus 1$ may consist 
of more than one element. 
For example, $\S$ is idempotent, $\K$ is not idempotent 
(because $1\tplus 1=\{0,1\}$ in $\K$), but both have
C-characteristic 1.

The characteristic of an idempotent multiring is 0, because in it
$1\tplus\dots\tplus1=1$ for any number of summands. In particular,
$\chr\S=0$.

A fundamental importance of the characteristic in the theory of rings
comes from the fact that the characteristic determines the minimal 
subring of the ring. For multirings no structural theorem of this sort is
known. 

Moreover, there is no commonly accepted notion of submultiring or
even subhyperfield. The point of disagreement is whether to require that
the subset underlying a submultiring would be closed under the 
multivalued addition, or just require that
the intersection of the subset with the sum of any of its two elements would
be non-empty. In the univalued situation there is no difference between
these two requirements. 

If we accept the alternative in which the subset contains the whole sum of
any two of its elements, then there is no hope for a reasonable list
of simple hyperfields. 

Under the other alternative, I am not aware about any conjectural  
list of simple hyperfields. However, the following two simple results in this
direction sounds inspiring. 

\begin{Th}\label{Th-simple-ch0-cch1}
In any multiring $R$ with $\chr R=0$  and $\cchr R=1$
the set $\{-1,0,1\}$ inherits from $R$ operations identical to the 
hyperfield operations in the sign hyperfield $\S$.\qed 
\end{Th}   

\begin{Th}\label{Th-simple-ch2} In any multiring $R$ with $\chr R=2$
the set $\{0,1\}$ inherits from $R$ operations identical to the operations
either in the Krasner hyperfield $\K$, if $\cchr R=1$, or in the field 
$\F_2$, if $\cchr R=2$.\qed
\end{Th}

\subsection{Hyperfields from a linearly ordered group}\label{s4.5} 
Let $X$ be a 
multiplicative group 
with a  linear order $\prec$ such that if $a\prec b$, then $ac\prec bc$ for
any $a,b,c\in X$. Let $Y$ be $X\cup\{0\}$. Extend the order $\prec$ from
$X$ to $Y$ by setting $0\prec x$ for any $x\in X$. 

The linear order $\prec$ gives rise to two multigroup structures in $Y$, 
with additions $\yplus$ and $\upspoon$, defined in Section \ref{s3.4}.
 Extend the multiplication in the group $X$ to $Y$ by defining $x0=0$ 
for each $x\in Y$.
It is easy to see that $Y$ with any of the additions, either $\yplus$ or
$\upspoon$, and this
multiplication is a doubly distributive hyperfield.

The Krasner hyperfield $\K$ can be obtained via this construction 
with $\yplus$ applied to the trivial group. The construction with
$\upspoon$ applied to the trivial group gives the field $\F_2$.

The sign hyperfield $\S$ cannot be presented as a hyperfield of a linear
order, because $\S$ is idempotent,
while any hyperfield of linear order is of characteristic 2: indeed,
$0\in 1\upspoon1\subset1\yplus 1$.

\subsection{Multiring homomorphisms}\label{s4.6}
Let $X$ and $Y$ be multirings. A map $f:X\to Y$ is called a {\sfit 
(multiring) homomorphism\/} if it is a multigroup homomorphism for the
additive multigroups of $X$ and $Y$ and a multiplicative homomorphism 
for their multiplicative semi-groups (the latter means that
$f(ab)=f(a)f(b)$ for any $a,b\in X$. A multiring homomorphism is said to be
{\sfit strong\/} if it is strong as a multigroup homomorphism for the
additive multigroups. 

There are many well known commonly used maps which are multiring 
homomorphisms. Below we consider a few examples.

\subsection{The sign homomorphism}\label{s4.7} The sign function 
$$\R\to\{0,\pm1\}:x\mapsto
\begin{cases} \frac{x}{|x|}, &\text{ if }x\ne0\\ 
  0,  &\text{ if }x=0\end{cases} 
$$ 
is a multiring homomorphism of the field $\R$ to the hyperfield $\S$.
For generalizations of this, see Marshall \cite{Marshall}.

\subsection{Ideals of a multiring}\label{s4.8} 
As in a ring, an {\sfit ideal\/} 
of a commutative multiring $X$ is a non-empty subset $I\subset X$ such that 
$a\tplus b\subset I$ for any $a,b\in I$, and $ab\in I$ 
if $a\in X$ and $b\in I$. For any multiring homomorphism $f:X\to Y$, 
its kernel $\Ker F=\{a\in X\mid f(a)=0\}$ is an ideal in $X$.  

As in the ring theory, for any ideal $I$ of a multiring $X$
one can construct the quotient $X/I$, and a multiring structure in $X/I$ 
such that the projection $X\to X/I$ is a strong multiring homomorphism. Any
multiring homomorphism $f:X\to Y$ admits a natural factorization 
$X\to X/\Ker f\to Y$. If $f$ is surjective and strong, then the induced
multiring homomorphism is an isomorphism. 

The assumption that $f$ is strong is necessary here. 
Without this assumption, a multiring homomorphism with a trivial kernel 
may be non injective.
On the other hand, most of interesting multiring homomorphisms 
are not strong. This is a major new phenomenon distinguishing multirings
from rings, cf \ref{s3.9}.

The example $\S\to \K:1,-1\mapsto1,0\mapsto0$ of a non-injective multigroup
homomorphism with trivial kernel considered in Section \ref{s3.9} above,
is in fact a multiring homomorphism of the sign hyperfield $\S$ to the
Krasner hyperfield $\K$ with the hyperfield structures defined in Section
\ref{s4.4}.

The sign homomorphism $\R\to \S$ defined in Section \ref{s4.7} above is
also a non-injective multiring homomorphism with trivial kernel.  

In a hyperfield $X$, the only ideals are $\{0\}$ and $X$. 

\subsection{Multiplicative kernel}\label{sn4.10}
The kernel does not contain all the information about a multiring 
epimorphism, in contrast to the ring theory. 
On the other hand, there exists a multiring epimorphism
that is not an isomorphism even if both the multirings
involved are hyperfields. 

If $f:X\to Y$ is a multiring homomorphism, and $X$ is hyperfields,
then either $\Ker f=0$ or $f=0$. 
Indeed, any ideal of a hyperfield $X$ is either $0$ or $X$ exactly for the
same reasons as if $X$ was a field. 

A hyperfield belongs to the traditional algebra at least in its
multiplicative structure. In a hyperfield the complement of the zero is a
commutative group. A non-trivial multiring homomorphism between 
hyperfields is a group homomorphism of the multiplicative groups.
As such, it has a kernel, the preimage of unity. 

In the univalued algebra, preimages of any two elements under a ring 
homomorphism are cosets related by translations which map bijectively one
of them onto another. In multivalued algebra
this phenomenon has no analogue. The formula $x\mapsto x+a$ defining a 
translation by $a$ in a ring, in a multiring 
turns into $x\mapsto x\tplus a$ which defines a multivalued map.
This map restricted to a preimage $f^{-1}(b)$ of an element $b$ 
under a multiring homomorphism $f:X\to Y$ does not send it to the
preimage of an element, but to the preimage of a set $b\tplus f(a)$, 
and this restriction is not invertible.
So, everything is broken.

If $X$ and $Y$ are hyperfields and $f:X\to Y$ is a multiring
homomorphism, then nonempty preimages of any non-zero element  
$b\in Y$ is related via natural bijections, which are multiplicative
translations, with $f^{-1}(1)$. 
For any $\Gb\in f^{-1}(b)$ formula
$x\mapsto \Gb^{-1}x$ maps $f^{-1}(b)$ onto $f{-1}(1)$, and this map has 
inverse $x\mapsto\Gb x$. 
The set $f^{-1}(1)$ is the kernel of the group homomorphism 
$X\sminus0\to Y\sminus0$ induced by $f$. Denote this kernel by $\Ker_mf$
and call it the {\sfit multiplicative kernel\/} of $f$. 
Obviously, $\Ker_mf$ is a subgroup of the multiplicative group of $X$.

Some fragments of this nice picture take place in a more general setup,
when $X$ and $Y$ are multirings and $f:X\to Y$ is a multiring homomorphism.
Still the multiplicative kernel $\Ker_mf$ is defined as $f^{-1}(1)$.
This set is obviously closed under multiplication, but may be not a
subgroup. Let $b\in f(X)$ and
$\Gb\in X$ such that $f(\Gb)=b$. Then multiplication by $\Gb$ maps
$\Ker_mf$ to $f^{-1}(b)$. However, as $\Gb$ may be non-invertible,
the construction for the inverse map $f^{-1}(b)\to\Ker_mf$ is not 
available. Moreover, simple examples show that the map $x\mapsto\Gb
x:\Ker_mf\to f^{-1}(b)$ may be neither injective nor surjective.

Elements $\Gb,\Gg\in X$ have the
same image under a map $f$ with given $\Ker_mf$ if there exist
$s,t\in\Ker_mf$ such that $s\Gb=t\Gg$. This is the weakest sufficient
conditions, which can be formulated solely in terms of $\Ker_mf$.
However, this is not a necessary condition.
 
\subsection{Multiplicative factorization}\label{sn4.11}
Any subgroup $S$ of the multiplicative group of a hyperfield $X$ can be
presented as the multiplicative kernel of a multiring homomorphism of 
$X$ to a hyperfield. A construction of this hyperfield was proposed
by Krasner \cite{Krasner2}, see also Marshall \cite{Marshall} and 
\cite{Marshall2}. 

The resulting hyperfield
is denoted by $X/_mS$. As a set, this is $(X^{\times}/S)\cup\{0\}$, 
a disjoint
union the zero and the quotient of the multiplicative group $X^{\times}$
by the subgroup $S$. The multiplication in $X/_mS$ is defined by the
multiplication in the quotient group and the identity $x0=0$. The addition
in $X/_mS$ induced by the addition in $X$. For cosets $aS,bS\in
X^{\times}/S$ the sum is $\{cS\mid c\in aS\tplus bS\}$, where $\tplus$
denotes the addition of subsets of $X$ induced by the addition in $X$.

The natural map $X\to X/_mS$ is a multiring homomorphism with
multiplicative kernel $S$.\smallskip  

\noindent{\bfit Examples. 1.} $X/_m(X\sminus\{0\})=\K$ for any 
hyperfield $X\ne\F_2$. \\
{\bfit 2.} $\R/_m\R_{>0}=\S$.\medskip

Marshall \cite{Marshall}, Example 2.6  
introduced the multiplicative factorization for  
more general situation in which $X$ is an arbitrary multiring and $S$ 
any subset of $X$ closed under multiplication. 
Then $X/_mS$ is a multiring obtained as the set of equivalence classes 
for the following equivalence relation: $a\sim b$ if there exist 
$s,t\in S$ such that $sa=tb$. If $0\in S$, then $X/_mS=0$.

Marshall's papers \cite{Marshall}, \cite{Marshall2} contain numerous
interesting applications of this construction. We restrict here to a simple
elementary example that was not considered in these papers. 

In a ring $\Z$ of integers, let $S$ be the set of all
odd numbers. Then $\Z/_mS$ can be identified with the set $\{2^n\mid
n=0,1,2,\dots\}$ of powers of 2. The multiplication in this multiring is
the usual multiplication of powers of $2$ (i.e., addition of the exponents).
The multivalued addition is the strict linear order operation $\upspoon$
from Section \ref{s3.4} for the order opposite to the ordering $<$ (i.e.,
$2^p\prec2^q$ iff $p>q$). This operation addresses to the following 
question: given two powers of 2, 
what is the highest power of 2 that can divide a sum of two
integers $m$ and $n$ for which the highest powers of 2 that divide
$m$ and $n$ are the given powers of 2. Clearly, $\chr \Z/_mS=2$ and
$\cchr \Z/_mS=2$.

If in this example we would replace 2 by an odd prime number $p$, then
the strict linear order operation would be replaced by a non-strict one,
the characteristic of the multiplicative quotient would be still 2,
and the C-characteristic would change to 1.  

\subsection{Prime ideals and homomorphisms to $\K$}\label{s4.9}
Cf. \cite{Marshall} Section 2.8 and \cite{CC1} Proposition 2.9. 
An ideal $I$ of a multiring is said to be {\sfit prime\/} 
if $1\ne I$ and $ab\in I$ implies that either $a\in I$ or $b\in I$. 

Notice that the kernel of any multiring homomorphism $f:X\to \K$ is a
prime ideal in $X$. Vice versa, any prime ideal can be presented in this
way. Indeed, for any prime ideal $I$ of a multiring $X$, define 
$$f_I:X\to \K:x\mapsto
\begin{cases} 0, &\text{ if }x\in I,\\
1, &\text{ if }x\not\in I.
\end{cases}
$$
This gives a multiring interpretation of prime ideals in usual rings.
Thus, the prime ideal spectrum $Spec K$ of a multiring $K$ can be 
identified with the set of multiring homomorphisms $K\to \K$.  

\section{Hyperfields from triangle inequalities}\label{s5}

\subsection{Triangle hyperfield}\label{s5.1} 
In the set $\R_+$ of non-negative
real numbers, define a multivalued addition $\nplus$ by formula
$$
a\nplus b=\{c\in\R_+\mid |a-b|\le c\le a+b\}.
$$
In other words, $a\nplus b$ is the set of all real numbers $c$ such that 
there exists an Euclidean triangle with sides of lengths $a,b,c$. 

\begin{Th}\label{th-triangle}
The set $\R_+$ with the multivalued addition $\nplus$ and usual
multiplication is a hyperfield.
\end{Th}

\begin{proof}
This addition is obviously commutative. It is also associative. In order to
prove this, just observe that both $(a\nplus b)\nplus c$ and
$a\nplus(b\nplus c)$ coincide with the set of real numbers $x$ 
such that there exists a Euclidean quadrilateral with sides of 
lengths $a,b,c,x$.\\ 
\centerline{$\vcenter
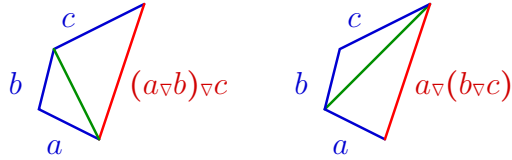
{\hbox{\begin{picture}(0,0)%
\includegraphics{figs/nplus-assoc.pstex}%
\end{picture}%
\setlength{\unitlength}{4144sp}%
\begingroup\makeatletter\ifx\SetFigFont\undefined%
\gdef\SetFigFont#1#2#3#4#5{%
  \reset@font\fontsize{#1}{#2pt}%
  \fontfamily{#3}\fontseries{#4}\fontshape{#5}%
  \selectfont}%
\fi\endgroup%
\begin{picture}(2557,986)(166,-305)
\put(406,-241){\makebox(0,0)[lb]{\smash{{\SetFigFont{12}{14.4}{\rmdefault}{\mddefault}{\updefault}{\color[rgb]{0,0,.82}$a$}%
}}}}
\put(181,119){\makebox(0,0)[lb]{\smash{{\SetFigFont{12}{14.4}{\rmdefault}{\mddefault}{\updefault}{\color[rgb]{0,0,.82}$b$}%
}}}}
\put(496,524){\makebox(0,0)[lb]{\smash{{\SetFigFont{12}{14.4}{\rmdefault}{\mddefault}{\updefault}{\color[rgb]{0,0,.82}$c$}%
}}}}
\put(2116,-241){\makebox(0,0)[lb]{\smash{{\SetFigFont{12}{14.4}{\rmdefault}{\mddefault}{\updefault}{\color[rgb]{0,0,.82}$a$}%
}}}}
\put(1891,119){\makebox(0,0)[lb]{\smash{{\SetFigFont{12}{14.4}{\rmdefault}{\mddefault}{\updefault}{\color[rgb]{0,0,.82}$b$}%
}}}}
\put(2206,524){\makebox(0,0)[lb]{\smash{{\SetFigFont{12}{14.4}{\rmdefault}{\mddefault}{\updefault}{\color[rgb]{0,0,.82}$c$}%
}}}}
\put(2611,119){\makebox(0,0)[lb]{\smash{{\SetFigFont{12}{14.4}{\rmdefault}{\mddefault}{\updefault}{\color[rgb]{.82,0,0}$a{\scriptscriptstyle{\nabla}}(b{\scriptscriptstyle\nabla}c)$}%
}}}}
\put(901,119){\makebox(0,0)[lb]{\smash{{\SetFigFont{12}{14.4}{\rmdefault}{\mddefault}{\updefault}{\color[rgb]{.82,0,0}$(a{\scriptscriptstyle\nabla}b){\scriptscriptstyle\nabla}c$}%
}}}}
\end{picture}%
}}$}
The usual multiplication is distributive over $\nplus$. The role of zero 
is played by $0$. The negation $a\mapsto -a$ for $\nplus$ is identity,
as for any $a\in\R_+$ the only real number $x$ such that 
$0\in a\nplus x$ is $a$.  
\end{proof}

This hyperfield is called the {\sfit triangle hyperfield\/} and denoted by
$\mftr$. 

\begin{Th}\label{th-non-ddistr-tr}
Hyperfield $\mftr$ is {\sfit not\/} doubly distributive. 
\end{Th}
\begin{proof}
Indeed, $2\nplus 1=[1,3]$.
Therefore $(2\nplus 1)\cdot(2\nplus 1)=[1,3]\cdot[1,3]=[1,9]$. 
On the other hand,
$$2\cdot 2\nplus 2\cdot1\nplus 1\cdot 2\nplus 1\cdot 1=4\nplus 2\nplus
2\nplus1$$ 
contains 0, because there exists an isosceles trapezoid with 
sides 4, 2, 1, and 2. In fact, $4\nplus 2\nplus
2\nplus1=[0,9]$. 
\end{proof}

The operation $\nplus$ appears in the representation theory. 
Denote by $V^{(a)}$ the $a$th irreducible representation of $sl_2\C$
(i.e., the symmetric power $\Sym^aV$ of the standard 2-dimensional 
representation $V$). Then the set $\{a\nplus b\}\cap (2\Z+a+b)$
parametrizes the set of irreducible representations of $sl_2\C$ 
which are the summands in $V^{(a)}\otimes V^{(b)}$:
$$
V^{(a)}\otimes V^{(b)}=\bigoplus_{c\in (a\nplus b)\cap (2\Z+a+b)}V^{(c)}
$$

\subsection{Ultratriangle hyperfield}\label{s5.2}
The construction of Section \ref{s4.5} of hyperfield of linearly ordered
group, 
when applied to the multiplicative group of positive real
numbers equipped with the usual order $<$, defines a structure of hyperfield 
in $\R_+$. Recall that the addition in this hyperfield is defined by 
formula
 $$
(a,b)\mapsto a\yplus b=
\begin{cases} 
\max(a,b), &\text{ if } a\ne b\\
\{x\in \R_+\mid x\le a\}, &\text{ if }a=b.
\end{cases}
$$ 
the multiplication is the usual multiplication of real numbers.
As any hyperfield of a linear order, this one is doubly distributive, see
Section \ref{s4.5}.

There is another way to construct the same hyperfield. It is completely
similar to the construction of the triangle hyperfield of Section
\ref{s5.1}, but with the triangle inequality replaced by the
non-archimedian (or ultra) triangle inequality $|c|\le\max(|a|,|b|)$. 
This hyperfield is called the {\sfit ultratriangle hyperfield\/} and 
denoted by $\mfutr$. 

\subsection{Tropical hyperfield}\label{s5.3}
The map $\log:\R_{>0}\to\R$ is naturally extended by mapping $0$ to
$-\infty$. The resulting map $\R_+\to\R\cup\{-\infty\}$ is denoted
also by $\log$. This is a bijection, and the hyperfield structure of
$\mfutr$ can be transferred via $\log$ to $\R\cup\{-\infty\}$. Denote the
resulting hyperfield by $\mftrop$, and call it the  {\sfit tropical hyperfield.\/}

The hyperfield structure of $\mftrop$ can be obtained by the construction
of Section \ref{s4.5} applied to the additive group of all real numbers with
the usual order $<$. The hyperfield addition here differs from 
the semifield addition $(a,b)\mapsto\max(a,b)$ in $\T$ only on the diagonal:
$\max(a,a)=a\ne a\yplus a=\{x\in\T\mid x\le a\}$, although
$\max(a,a)\in a\yplus a$. 

Since $\mftrop$ will play an important role in what follows, let me describe it
explicitly and independently of constructions above.
The underlying set of $\mftrop$ is $\R\cup\{-\infty\}$, the addition is
$$
(a,b)\mapsto a\yplus b=
\begin{cases} 
\max(a,b), &\text{ if } a\ne b\\
\{x\in \mftrop\mid x\le a\}, &\text{ if }a=b.
\end{cases}
$$ 
the multiplication is the usual addition of real numbers extended in the
obvious way to $-\infty$, the hyperfield zero is $-\infty$, the hyperfield
unity is $0\in\R$. 

\subsection{Amoeba hyperfield}\label{s5.4}
Transfer via the same bijection $\log:\R_+\to\R\cup\{-\infty\}$ 
the structure of the triangle hyperfield $\mftr$ defined above in 
Section \ref{s5.1} to $\R\cup\{-\infty\}$. The resulting hyperfield is called 
the {\sfit amoeba hyperfield\/} and 
denoted  by $\mfamb$. 

The addition in $\mfamb$ is defined by formula
$$
a\closedcurlyvee b=\{c\in\R\mid \log(|e^a-e^b|)\le c\le \log(e^a+e^b)\},
$$
while the multiplication in $\mfamb$ is the  usual addition.

\subsection{Multiplicative seminorm}\label{s5.5} 
Let $K$ be a ring. Recall that a map $K\to\R_+:x\mapsto|x|$ is a 
{\sfit multiplicative seminorm\/} if $|x+y|\le|x|+|y|$ and $|xy|=|x||y|$ 
for any $x,y\in K$. Obviously, a multiplicative seminorm is nothing but a
multiring homomorphism of $K\to\mftr$.

\subsection{Non-archimedian multiplicative seminorm}\label{s5.6}
Recall that a multiplicative seminorm $K\to\R_+:x\mapsto|x|$ is 
non-archimedian if $|x+y|\le\max(|x|,|y|)$. A non-archimedian
multiplicative seminorm $K\to\R_+$ is a multiring homomorphism 
$K\to\mfutr$. A non-archimedian valuation map, that is a composition of 
a non-archimedian multiplicative seminorm with 
$\R_+\to\mftrop:x\mapsto \log x$, is a multiring homomorphism $K\to\mftrop$.

\section{Tropical addition of complex numbers}\label{s6} 

\subsection{Definition}\label{s6.1}
The {\sfit tropical sum\/} $a\cplus b$ of arbitrary complex numbers $a$ and $b$ is defined as follows. 
\begin{figure}[htb]
\centerline{$\input{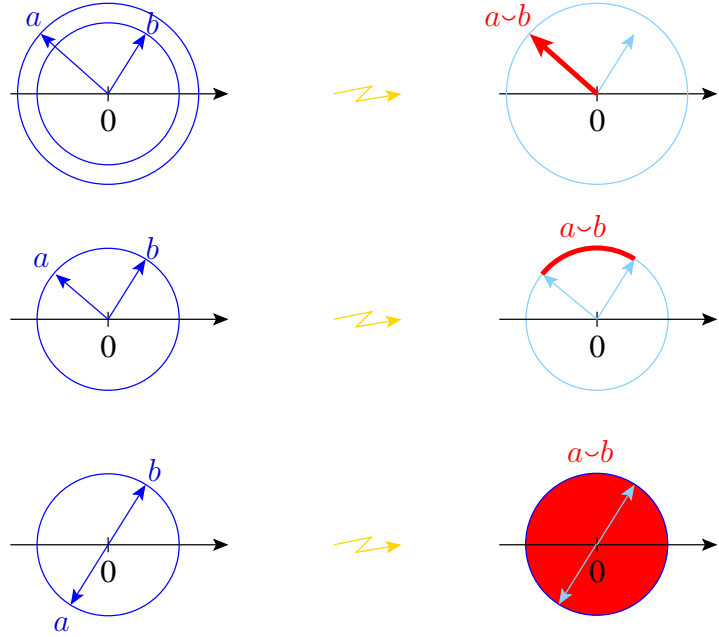}$}
\caption{Tropical addition of complex numbers.}
\label{tplusC}
\end{figure}
\begin{itemize} 
\item
If $|a|>|b|$, then $a\smallsmile b=a$. 
\item
If $|a|<|b|$, then $a\cplus b=b$. 
\item 
If $|a|=|b|$ and $a+b\ne0$, then $a\cplus b$ is the {\sfit set\/} of
all complex numbers which belong to the shortest arc connecting $a$ with $b$ 
on the circle of complex numbers with the same absolute value.\\ 
In formulas:
if $a=re^{\Ga i}$, $b=re^{\Gb i}$ with $|\Gb-\Ga|<\pi$, then 
$a\cplus b=\{re^{\Gf i} \mid|\Ga-\Gf|+|\Gf-\Gb|=|\Ga-\Gb|\}$.
\item 
If $a+b=0$, then $a\cplus b$ is the whole closed disk 
$\{c\in\C \mid|c|\le|a|\}$.
\end{itemize}

\subsection{Obvious properties}\label{s6.2}
The tropical addition is {\sfit commutative,\/} \\ $a\cplus b=b\cplus a$ for
any $a,b\in\C$. 
This follows immediately from the definition.

The zero plays the same role of the neutral element as it plays for the usual
addition: $a\cplus 0=a$ for any $a\in\C$.

Furthermore, for any complex number $a$ there is a unique $b$ such that
$0\in a\cplus b$. This $b$ is $-a$.  

\subsection{Associativity}\label{s6.3}
\begin{Th}\label{th1}
The tropical addition of complex numbers is associative.
\end{Th}

A straightforward proof is elementary, but quite
cumbersome. It is postponed till Appendix 1.  

\subsection{Distributivity.}\label{s6.4}
\begin{Th}\label{Th2}
The usual multiplication of complex numbers is distributive over the 
tropical addition: $a(b\cplus c)=ab\cplus ac$ for any complex numbers $a$,
$b$ and $c$. 
\end{Th}

Indeed, all the constructions and characteristics of summands involved in the
definition of tropical addition are invariant under multiplication by 
a complex number: the ratio of absolute values of two complex numbers is
preserved, an arc of a circle centered at $0$ is mapped to an arc of a
circle centered at $0$, a disk centered at $0$ is mapped to a disk centered
at $0$.\qed

\begin{Th}\label{th-non-dd-tc}
The multiplication of complex numbers is not doubly distributive
over the tropical addition.
\end{Th}

\begin{proof}
Compare $(1\cplus i)(1\cplus -i)$ with $1\cdot1\cplus 1\cdot -i\cplus
i\cdot1\cplus i\cdot(-i)=1\cplus i\cplus -i\cplus 1$.

Since $1\cplus i$ is the arc of the unit circle connecting $1$ and $i$,
and $1\cplus -i$ is the arc of the unit circle connecting $1$ and $-i$,
their (pointwise) product is the arc of the unit circle connecting $i$ and
$-i$. 
On the other hand, the tropical sum $1\cplus i\cplus -i\cplus 1$ is the
whole unit disk. 
\end{proof}

\subsection{Complex tropical hyperfield}\label{s6.5}
Thus, the set $\C$ of complex numbers with the tropical addition and usual
multiplication is a hyperfield. Denote it by $\tc$ and call {\sfit 
complex tropical hyperfield.\/}

\subsection{The tropical sum of several complex numbers}\label{s6.6}
The tropical sum of several complex numbers is affected only by those 
summands which have the greatest absolute value. A summand whose 
absolute values is not maximal does not contribute at all. 

\begin{Th}\label{th-1}
Let $a_1$, \dots, $a_n$ be complex numbers with absolute values equal $r$. 
Then 
\begin{itemize} 
\item either $a_1\cplus\dots\cplus a_n$ is the closed disk with radius $r$
centered at $0$, it can be obtained as the sum of at most three of the
summands  $a_1$, \dots, $a_n$  and $0\in\Conv(a_1,\dots,a_n)$,
\item or $a_1\cplus\dots\cplus a_n$ is contained in a half of the circle 
of radius $r$ centered at 0 and is the tropical sum of at most two of the
summands $a_1$, \dots, $a_n$ (so, it is either a point or a closed arc).
\end{itemize}
\end{Th}
 
The proof of Theorem \ref{th-1} is elementary and straightforward. See
Appendix 2.

\begin{cor}\label{cor-2}
The tropical sum of any finite set of complex numbers equals the tropical
sum of a subset consisting at most of three summands. If the tropical sum
does not contain the zero, then the number of summands can be reduced to
two.\qed
\end{cor}

\begin{cor}\label{cor-1}
The tropical sum of a finite set of complex numbers contains the zero iff 
the zero is contained in the convex hull of the summands having the
greatest absolute value.\qed 
\end{cor}

\section{Relations of $\tc$ with other hyperfields}\label{s7}

\subsection{Submultirings and subhyperfields}\label{s7.1}
\begin{Th}\label{th-sbmultirings}
Any subset $A$ of $\C$ containing $0$, invariant under the involution 
$x\mapsto-x$ 
and closed with respect to multiplication inherits the structure of 
multiring from $\tc$. 

If, furthermore, 
$A\sminus 0$ is invariant under the involution $x\mapsto x^{-1}$, then $A$
with the inherited structure is a hyperfield.
\qed
\end{Th}

In particular, any subfield of $\C$ inherits structure of
hyperfield from  $\tc$. 

\subsection{The tropical real hyperfield $\tr$}\label{s7.2}
For example, $\R$ inherits the structure of hyperfield. The induced addition 
$(a,b)\mapsto a\cplus_{\R}b=(a\cplus b)\cap\R$ can
be described directly as follows:$$
a\cplus_{\R} b = 
\begin{cases}
\{a\}, &\text{ if }\quad |a|>|b|,\\
\{b\}, &\text{ if }\quad |a|<|b|,\\
\{a\}, &\text{ if }\quad  a=b,\\
[-|a|,|a|], &\text{ if }\quad a=-b.
\end{cases}
$$
\begin{figure}[bht]
\centerline{$\input{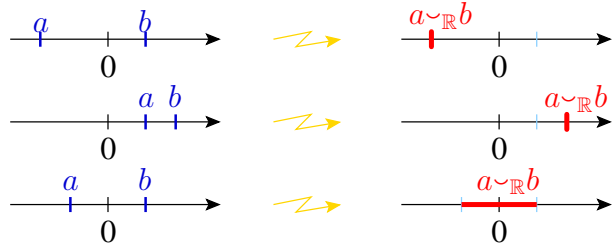}$}
\caption{Tropical addition of real numbers.}
\label{tplusR}
\end{figure}
See also Figure \ref{tplusR}. 

The operation $(a,b)\mapsto a\cplus_{\R}b$
is called the  {\sfit tropical real addition\/} or even just  {\sfit
tropical addition,\/} when there is no danger of confusion.   
The set $\R$ with the tropical real addition and usual multiplication is
called the {\sfit tropical real hyperfield\/} and denoted by $\tr$. 
 
\begin{Th}\label{Th-dd-tr} $\tr$ is doubly distributive.
\end{Th}

\begin{proof}
The only situation in which double distributivity 
$$(a\cplus_{\R}b)(c\cplus_{\R}d)=ac\cplus_{\R}ad\cplus_{\R}bc\cplus_{\R}bd$$
is not obvious, is when both factors in the left hand side consist of more
than one element. Then $a=-b$ and $c=-d$ and both the left hand side 
and right hand side equal
$[-|ac|,|ac|]$.
\end{proof}

\subsection{The hyperfield of signs}\label{s7.3}
A structure of
hyperfield comes in the same way to subsets $A\subset\C$ which are not
subfields of $\C$. For example, $\{-1,0,+1\}\subset\C$ satisfies the
conditions of Theorem \ref{th-sbmultirings}, and hence inherits a hyperfield 
structure from $\tc$. This hyperfield appeared above as $\S$. 

Recall that there is a multiring homomorphism 
$$\R\to\{0,\pm1\}:x\mapsto
\begin{cases} \frac{x}{|x|}, &\text{ if }x\ne0\\ 
  0,  &\text{ if }x=0\end{cases} 
$$ 
 of the field $\R$ to the hyperfield $\S$.

This map is also a multiring homomorphism $\tr\to \S$.

\subsection{The phase hyperfield}\label{s7.4}
Let $\Phi$ be $\{z\in\C\mid |z|=1\}\cup\{0\}$, that is the unit circle in
$\C$ united with its center. This set satisfies the
conditions of Theorem \ref{th-sbmultirings}, and, by that theorem, $\Phi$
inherits a hyperfield structure from $\tc$. It will be called
the {\sfit phase hyperfield\/} and denoted by $\Phi$. One can obtain 
$\Phi$ also as $\C/_m\R_{>0}$.
Notice that
$\S=\Phi\cap\R$. 

The map 
$$\C\to\Phi: z\mapsto
\begin{cases} \dfrac{z}{|z|}, &\text{ if }z\ne0\\
0, &\text{ if }z=0
\end{cases}
$$
is called the {\sfit phase map\/}. This is a multiring homomorphism in two
senses: $\C\to\Phi$ and  $\tc\to\Phi$. 

\subsection{Embedding  $\T\subset\tc$}\label{s7.5}
Recall that a {\sfit semifield} is a set with two (univalued) operations, 
addition and multiplication, which satisfy all the axioms of field, 
except that there is no subtraction. 

A classical example of a semifield is the set $\R_+$ of
non-negative real numbers with the usual addition and multiplication. 
Another semifield structure in the same set is defined by replacing the
usual addition with the operation of taking the greatest of two numbers:
$(a,b)\mapsto\max(a,b)$. 

There is an isomorphism of the tropical semifield 
$\T$  onto the semifield $\R_{\ge0, \max,\times}$ mapping $x\mapsto \exp
x$ for $x>0$, and $-\infty\mapsto 0$. 

Observe that the semifield addition $(a,b)\mapsto\max(a,b)$ in 
$\R_+$ is induced from 
the addition in $\tc$ (or $\tr$, does not matter).
Indeed,  $a\cplus b=\max(a,b)$  for any $a,b\in\R_+$.  

Thus, the semifield $\R_{\ge0,\max,\times}$ is a subset of the hyperfield  
$\tc$ closed with respect to both binary operations of $\tc$, 
and the binary operations 
coincide with the operations of the semifield $\R_{\ge0,\max,\times}$.
In particular, the inclusion $\R_{\ge0,\max,\times}\to\tc$ and its
composition $\T\to\tc$ with the isomorphism $\T\to\R_{\ge0,\max,\times}$
are homomorphisms.

{\bf Warning.} There is a natural map in the opposite direction  
$\tc\to\R_+:z\mapsto|z|$. It is a right inverse for the inclusion.
However, this is not a homomorphism for the tropical addition $\cplus$. 
Indeed, $x\cplus(-x)\cap\R_+=[0,|x|]$ for any $x\in\R$,  
but $|x|\cplus|-x|=|x|$, which does not contain $[0,|x|]$ for $x\ne0$.

In order to make the map $\tc\to\R_+:z\mapsto|z|$ a homomorphism,
one should consider a hyperfield structure in $\R_+$. 

\subsection{The absolute value and amoeba maps}\label{s7.6}
The map $\C\to\R_+:z\mapsto|z|$ is also a homomorphism from many
points of view. This is
\begin{itemize} 
\item  a multiring homomorphism $\C\to\mftr$ 
from the field of complex numbers to the triangle hyperfield (see Section
\ref{s5.1});
\item  a multiring homomorphism $\tc\to\mftr$ 
from the complex tropical hyperfield $\tc$ to the triangle hyperfield;
\item a multiring homomorphism $\tc\to \mfutr$ from $\tc$ to the
ultratriangle hyperfield (see Section \ref{s5.2});
\end{itemize}
The composition of this map with $\log:\R_+\to\R\cup\{-\infty\}$
is a multiring homomorphism 
\begin{itemize} 
\item $\C\to \mfamb$;
\item $\tc\to \mfamb$;
\item $\tc\to\mftrop$.
\end{itemize}

\subsection{Complex polynomials and $\tc$}\label{s7.7}
The map $w$ which is defined and discussed in this section and the next
one, essentially  was defined by
Mikhalkin  \cite{Mikha} and used him in his
definition of complex tropical curves. However, the algebraic properties of
$w$ were not considered, because the tropical addition of complex numbers
was not available. 

Let $p(X)\in\C[X]$ be a polynomial in one variable  $X$ with complex
coefficients, $p(X)=\sum_{k=0}^na_kX^k$, where $a_k\in\C$, $a_n\ne0$. 
Let $w(p)=\frac{a_n}{|a_n|}e^n$. Further, let $w(0)=0$.
This defines a map $\C[X]\to\C:p\mapsto w(p)$. 

\begin{Th}\label{th0}
The map $w$ is a multiring homomorphism of the polynomial ring  $\C[X]$ to 
the hyperfield  $\tc$, that is   
$w(p+q)\in w(p)\tplus w(q)$ and
$w(pq)=w(p)w(q)$ for any $p,q\in \C[X]$.
\end{Th}

\begin{proof}[{\bfseries Proof.\/}]
The value of $w$ on a polynomial $p$ is equal to the value of $w$ on the
monomial of $p$ having the greatest degree. For a monomial $p(X)=aX^n$ the
value of $w$ equals $\dfrac{p(e)}{|p(1)|}$. Obviously, the latter formula 
defines a multiplicative homomorphism. 

Let us prove that $w(p+q)\in w(p)\tplus w(q)$ for any $p,q\in \C[X]$.
Let the highest degree monomials of $p$ and $q$ are $aX^n$ and $bX^m$,
respectively (so that $\deg p=n$, $\deg q=m$). 
If $n>m$, then the highest degree term of  $p+q$ equals $aX^n$ and
$w(p+q)=w(p)=w(p)\cplus w(q)$. Similarly, if $n<m$, then 
$w(p+q)=w(q)=w(p)\cplus w(q)$. 

If the degrees of $p$ and $q$ are the same, and the coefficients $a$ and $b$ 
of their monomials of the highest degree are such that
$\frac{a}{|a|}+\frac{b}{|b|}\ne0$, then these monomials do not annihilate
each other in the sum, and the monomial of highest degree of $p+q$ 
is the sum of these monomials. Its degree equals $\deg p=\deg q$, the
coefficient is $a+b$. However, the argument $\frac{a+b}{|a+b|}$ of this 
coefficient is not determined by $\frac{a}{|a|}$ and $\frac{b}{|b|}$.
It can take any value in the open intervale between the arguments of the 
summands. In particular, it takes values in the set of arguments of complex
numbers belonging to $w(p)\tplus w(q)$.

If $\deg p=\deg q$ and the coefficients $a$ and $b$ of the highest terms
are such that  $\frac{a}{|a|}+\frac{b}{|b|}=0$, then the highest terms may
annihilate under summation. Therefore the highest term of $p+q$ is either
equal to the sum of the highest terms of  $p$ and $q$,
or come from terms of lower degrees and cannot be recovered from the terms
of the highest degree. The only that we can say about it if we know 
only $w(p)$ and $w(q)$ (i.e., if we know only the arguments of the 
coefficients in the terms of the highest degrees and the degrees), 
is that its degree is not greater than the degree of the summands. This 
implies  $w(p+q)\in w(p)\tplus w(q)$.   
\end{proof}

\subsection{Real exponents}\label{s7.8}
The image of $w$ consists of only those complex numbers whose absolute
values are powers of  $e$. However similar constructions are able to
provide multiring homomorphisms onto the whole $tc$. For this, it is enough
to replace usual polynomials by polynomials with arbitrary real exponents.

Let us replace $\C[X]$ by the group algebra $\C[\R]$ of the additive group
$\R$. Elements of $\C[\R]$ can be thought of as $\sum_ka_kX^{r_k}$, 
where $a_k\in\C$, $r_k\in\R$. The formal variable $X$ symbolizes here the
transition from additive notation for addition in $\R$ to multiplicative
notation in $\C[\R]$, where additive notation is reserved 
for the formal sum.

Elements of  $\C[\R]$ may be interpreted as functions  $\C\to\C$.
For this, let us turn   $\sum_ka_kX^{r_k}$ into an exponential sum 
$\sum_ka_ke^{r_kT}$ by replacing $X$ with $e^T$.

The map $w:\C[X]\to\C$ extends to $\C[\R]$ as follows: choose from the sum  
$\sum_ka_kX^{r_k}$ the summand with the greatest exponent, say, 
$a_nX^{r_n}$ and apply the same formula to it $\frac{a_n}{|a_n|}e^{r_n}$. 
The map is a multiring homomorphism of the ring  $\C[\R]$ onto the
hyperfield $\tc$. The proof that this is a multiring homomorphism is
literally the same as the proof of Theorem \ref{th0} above. 

A ring can be replaced here by an algebraically closed field real-power
Puiseux series $\sum_{r\in I}a_rt^r$, where $I\subset\R$ is a well-ordered
set. Cf. Mikhalkin \cite{Mikha}, Section 6.

This construction demonstrates how one can obtain the tropical addition of
complex numbers from the usual addition of polynomials. It is clear why it
should be multivalued. For complex numbers $a$ and 
$b$ with $|a|=|b|$, but  $a\ne-b$ any $c$ for the open arc 
$(a\tplus b)\sminus\{a,b\}$, one can find $A,B,C\in\C[\R]$ such that 
$w(A)=a$, $w(B)=b$ and $w(C)=c$, see
Figure \ref{sumCoe}. Complex numbers $a,b\in\C$ with $a+b=0$ are
represented as the images under $w$ of polynomials $A,B\in\C[\R]$ 
with highest degree terms opposite to each other and annihilating under
addition of the polynomials. The highest degree term of $A+B$
is not controlled by the highest degree terms of the summands $A$ and $B$, 
but its degree does not exceed the degree of the summands. 

\begin{figure}
\centerline{$\input{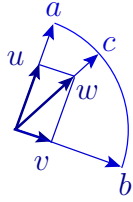}$}
\caption{For given $a,b,c\in\C$, constructing  $u,v,w\in\C$
such that 
$A=uX^{r}$, $B=vX^r$, $C=wX^r$ with $r=\log|a|$ and $w(A)=a$, $w(B)=b$,
$w(C)=c$, $A+B=C$. }
\label{sumCoe}
\end{figure}

\section{Continuity}\label{s8}

In the set of all subsets of a topological space, there are various natural
topological structures. However none of them is perfect. The most classical
of them are three structures introduced by Vietoris \cite{Vietoris} in 
1922. The multivalued additions considered above are continuous with 
respect to one of them, the upper Vietoris topology, and this implies
important properties of multivalued functions defined by polynomials over
these hyperfields.

\subsection{Vietoris topologies}\label{s8.1}
The {\sfit upper Vietoris topology\/} in the set $2^X$ of all subsets of a
topological space $X$ is the topology generated by the sets 
$2^U\subset2^X$, where $U$ is open in $X$. A neighborhood of a set 
$A\subset X$ in the upper Vietoris topology should contain all subsets 
of a set $U$ that is open in $X$ and contains $A$.

This topology is quite odd. For example, it is far from being Hausdorff:
sets with non-empty intersection cannot have disjoint neighborhoods in it.
Therefore usually a limit in the upper Vietoris topology is not unique.
By adding new points to a limit we would get a limit. Probably this is what 
motivates the word {\sfit upper\/} in the name of the topology.

The {\sfit lower Vietoris topology\/} in the set $2^X$ of all subsets of a
topological space $X$ is the topology generated by the sets 
$2^X \sminus 2^{C}$, where $C$ is a closed subset of $X$. In other words,
the lower Vietoris topology is generated by sets 
$\{Y\subset X\mid Y\cap U\ne\varnothing\}$, where $U$ is an open set
of $X$. In the lower Vietoris topology, closed sets are generated by closed
sets of $X$ in the most direct way: a closed set $C\subset X$ 
gives rise to the set $2^C\subset 2^X$ closed in the lower Vietoris
topology. Recall that in the upper Vietoris topology open sets are
generated similarly by open subsets of $X$. A neighborhood of a set 
$A\in 2^X$  in the lower Vietoris topology should contain all sets
intersecting with open sets $U_1,\dots,U_n\subset X$ which meet $A$. 
A limit in the lower Vietoris topology also usually is not unique, but for
the opposite reason: it would stay a limit under removing of its points.

The topology generated by the upper and lower Vietoris topologies is called 
just the {\sfit Vietoris topology.\/}
  
\subsection{Continuity and semi-continuities}\label{s8.2}
A multimap $X\multimap Y$ is said to be  
\begin{itemize} 
\item {\sfit upper semi-continuous\/} if the corresponding map 
$f^\uparrow:X\to2^Y$ is continuous with respect to the upper Vietoris
topology in  $2^Y$;
\item {\sfit lower semi-continuous\/} if  
$f^\uparrow:X\to2^Y$ is continuous with respect to the lower Vietoris
topology in  $2^Y$;
\item {\sfit continuous\/} if  
$f^\uparrow:X\to2^Y$ is continuous with respect to the Vietoris topology in 
 $2^Y$ (i.e., $X\multimap Y$ is both upper and lower semi-continuous).
\end{itemize}

Recall that the set $\{a\in X\mid f(a)\subset B\}$ 
is called the {\sfit upper preimage\/} of  $B$ under $f$, and the set 
$\{a\in X\mid f(a)\cap B\ne\varnothing\}$ is called the  {\sfit lower
preimage\/} of $B$ under $f$. 

It is easy to see that $f:X\multimap Y$ is upper (respectively, lower)
semi-continuous if and only if the upper (respectively, lower) preimage of
any set open in $Y$ is open in $X$. 

\subsection{Tropical additions}\label{s8.3}
\begin{Th}\label{th2}
The tropical addition $\C\times\C\multimap{\C}:(a,b)\mapsto a\cplus b$ is
not lower semi-continuous 
(i.~e., the corresponding map $\C\times\C\to2^{\C}$ is not continuous with
respect to the classical topology in $\C^2$ and the lower Vietoris topology in 
$2^{\C}$).
\end{Th}

\begin{proof}
It would suffice to find a set such that it is open in the lower Vietoris 
topology and its preimage is not open in the classical topology of
$\C^2$. Take, for instance, the set $H$ consisting of sets $A$ meeting 
the open disk of radius 1 and center 0. Its preimage is the set of pairs
$(a,b)$ of complex numbers such that $a\cplus b$ meets the disk. The
preimage of $H$ consists of pairs $(a,b)$ which satisfy one of the 
following two conditions: either  $|a|<1$ and $|b|<1$, or $a=-b$. 
Obviously, this set is not open. 
\end{proof}

Similarly one can prove that the additions in the ultratriangle 
hyperfield $\mfutr$ (see Section \ref{s5.2}),  the tropical 
hyperfield $\mftrop$ (Section \ref{s5.3}) and the real tropical
hyperfield $\tr$ (Section \ref{s7.2}) are not lower semi-continuous.  

\begin{Th}\label{th3}
The tropical addition $\C\times\C\multimap{\C}:(a,b)\mapsto a\cplus b$
is upper semi-continuous (i.~e.,~the corresponding map
$\C\times\C\to2^{\C}$ is continuous with respect to the classical topology
in $\C^2$ and the upper Vietoris topology in $2^{\C}$).
\end{Th}

\begin{proof} Let us prove the corresponding local continuity, i.e., prove
that for any neighborhood $V\subset2^{\C}$ of the image $a\cplus b$ of 
$(a,b)$ there exists a neighborhood $U\subset\C^2$ of $(a,b)$ such that the
image of $U$ is contained in $V$. In the upper Vietoris topology, a base of 
neighborhoods of $a\cplus b$ is composed by sets $2^W$ where $W$ runs over
a base of neighborhoods of $a\cplus b$ in $\C$. Thus, it would suffice 
for an arbitrarily small neighborhood 
$W\supset a\cplus b$ to find a neighborhood $U$ of $(a,b)$ in 
$\C^2$ such that  $x\cplus y\subset W$ for any $(x,y)\in U$. 
Consider one by one each of the three kinds of $(a,b)$. 

If $|a|>|b|$, then $a\cplus b=a$. Any neighborhood of 
$a$ contains an open disk centered at $a$. Diminish it if needed in order
to ensure that its radius  $r$ is smaller than $\frac12(|a|-|b|)$. 
Choose for  $W$ an open disk $B_r(a)$ of radius $r$ centered at $a$.
Then for $U$ one can take the neighborhood $B_r(a)\times B_r(b)$ 
of $(a,b)$. Obviously, $B_r(a)\cplus B_r(b)= B_r(a)$.

If $|a|=|b|$ and $a+b\ne 0$, then $a\tplus b$ is the shortest arc $C$ 
connecting $a$ and $b$ in the circle centered at 0.
Let $r$ be a positive real number, which so small that the disks  $B_r(a)$
and $B_r(b)$ do not contain points symmetric to each other with respect to
0. Any neighborhood of $C$ in $\C$ contains $W=B_\rho(a)\cplus B_\rho(b)$
with some $\rho\in(0,r)$. Let $U= B_\rho(a)\times B_\rho(b)$. 

If $|a|=|b|$ and $a+b=0$, then $a\cplus b$ is the closed disk centered at 0 
with radius $|a|$. Any neighborhood of this disk in $\C$ contains a
concentric open disk of some radius $r>|a|$. Let $U$ be this disk. 
The image of $U\times U$ under the tropical addition is $U$. 
\end{proof}

Similarly one can prove that the tropical addition of real numbers is upper
semi-continuous.

\subsection{Continuity of triangle additions}\label{s8.4}
Recall that the triangle addition of non-negative real numbers is defined 
by formula $a\nplus b=\{c\in\R_+\mid |a-b|\le c\le a+b\}$.

\begin{lem}\label{lem-cont}
A multimap $f:X\multimap\R$ is continuous if there exist continuous functions 
$f_\pm:X\to\R$ with $f(x)=[f_-(x),f_+(x)]$ for any $x\in X$ and
$f_+(x)>f_-(x)$ on everywhere dense subset of $X$. \qed
\end{lem}

The triangle addition satisfies the hypothesis of Lemma, hence it is
continuous, (i.e., the corresponding map $\R_+\times\R_+\to2^{\R_+}$ is
continuous with respect to the classical topology in $\R_+\times\R_+$ the
the Vietoris topology in $2^{\R_+}$).

Similarly, from Lemma \ref{lem-cont} it follows that the addition in the 
amoeba hyperfield $\mfamb$ is continuous.

\subsection{Properties of upper semi-continuous multimaps}\label{s8.5} 
As we see above the additions in hyperfields $\tc$, $\tr$, $\mftr$, 
$\mfutr$, $\mftrop$ and $\mfamb$ are upper semi-continuous. Let $K$ denote
one of these hyperfields.
 
First, notice, that for univalued maps 
upper semi-continuity is equivalent to continuity. 

Second, obviously, a composition of upper semicontinuous maps is upper
semi-continuous.

From these two statements it follows immediately that a multimap defined by 
a polynomial over $K$ is upper semi-continuous.

\begin{Th}\label{th5}
Let $X,Y$ be topological spaces, $f:X\multimap Y$ be an upper 
semi-continuous multimap and $C\subset Y$ be a closed set. Then the set 
$\{a\in X\mid f(a)\cap C\ne\varnothing\}$ is closed.
\end{Th}

\begin{proof} The set  
$\{B\in2^Y\mid B\subset X\sminus C\}$ is open in the upper Vietoris
topology of $2^Y$. Therefore, due to upper semi-continuity of multimap 
$f$, the preimage $f^{+}=\{a\in X\mid f(a)\subset X\sminus C\}$ of this set
under the $f^\uparrow:X\to2^Y$ is open. Consequently, the set
$\{a\in X\mid f(a)\cap C\ne\varnothing\}=X\sminus \{a\in X\mid f(a)\subset
X\sminus C\}$ is closed.  
\end{proof}

\begin{cor}\label{cor-th5} For any polynomial $p$ over 
$K$, the set defined by condition $0\in p(x_1,\dots,x_n)$
is closed in $K^n$.\qed
\end{cor}

Finally, recall two well-known theorems about upper semi-continuous
multimaps.

\begin{Th}\label{th4}
The image of a connected set under an upper semi-continuous multimap 
is connected, if the image of each point is connected. \qed
\end{Th}

\begin{Th}\label{th8}
The image of a compact set under an upper semi-continuous multimap 
is compact, if the image of each point is compact.
\qed
\end{Th}

\begin{cor}\label{cor}
A multimap defined by a polynomial $p(x_1,\dots,x_n)$ over $K$  
maps connected sets to connected and compact sets to compact. In
particular, the graph of $p$ is connected.  \qed
\end{cor}

\section{Dequantizations}\label{s9}

\subsection{The Litvinov-Maslov dequantization}\label{s9.1}
Consider a family of semirings $\left\{S_h\right\}_{h\in[0,\infty)}$ 
(recall that a semiring is a sort of ring, but without subtraction).
As a set, each of $S_h$ is $\R$.
The semiring operations $+_h$ and $\times_h$ in $S_h$ 
are defined as follows:
\begin{align}a+_h b&=\begin{cases} h\ln(e^{a/h}+e^{b/h}),& \text{ if }h>0 \\
                            \max\{a,b\}, & \text{ if } h=0
\end{cases}\label{oplus}\\
a\times_h b&= a+b\label{odot}
\end{align}
These operations  depend continuously on $h$.
For each $h>0$ the map 
$$D_h:\R_{>0} \to S_h: x\mapsto h\ln x$$ 
is  a semiring isomorphism of $\left\{\R_{>0},+,\cdot\right\}$ onto
$\left\{S_h,+_h,\times_h\right\}$, that is
$$D_h(a+b)=D_h(a)+_hD_h(b),\qquad
D_h(ab)=D_h(a)\times_hD_h(b).
$$
Thus $S_h$ with $h>0$ can be considered as a copy of $\R_{>0}$
with the usual operations of addition and multiplication. On the other
hand, $S_0$ is a copy $\R_{\max,+}$ of $\R$, where the operation 
of taking maximum is considered as an addition, and the usual addition, 
as a multiplication.

Applying the terminology of quantization to this deformation, we must
call $S_0$ a classical object, and $S_h$ with $h\ne0$, quantum ones.
The whole deformation is called the {\sfit Litvinov-Maslov dequantization\/}
of positive real numbers. The addition in the resulting semiring 
$\R_{\max,+}$, is idempotent in the sense that $\max(a,a)=a$ for any $a$.  
 
The analogy with Quantum Mechanics motivated  the following \smallskip

\noindent{\bfit
Correspondence principle\/}  (Litvinov and Maslov \cite{L-M}).
{\sfit  ``There exists a (heuristic) correspondence, in the spirit of the
correspondence principle in Quantum Mechanics, between important,
useful and interesting constructions and results over the field of real
(or complex) numbers  (or the semiring of all nonnegative numbers) and 
similar constructions and results over idempotent semirings.''}

This principle proved to be very fruitful in a number of situations, see
\cite{L-M}, \cite{L-M-S}. The Litvinov-Maslov dequantization helps to 
relate the corresponding things.
 
Indeed, any valid formula involving only positive real 
numbers and only arithmetic operations survives under the limit and turns 
into a valid formula in $\R_{\max,+}$. 

The correspondence principle is formulated much wider, than this transition
to limit allows: not only for the semirings of all positive real numbers
and $\R_{\max,+}$, but for any idempotent semiring, on one hand, and the fields
$\R$ and $\C$, on the other hand. 

One may expect that there are extra mathematical reasons for this heuristic
correspondence. Below similar dequantization deformations are presented.
However, the dequantized objects are not semifields, but rather
mutlifields. 

\subsection{Dequantization of the triangular hyperfield to the ultra-triangular}\label{s9.2}
For a positive real number $h$, let $R_h:\R_+\to\R_+$ be the map
defined by formula $x\mapsto x^{\frac1h}$. 
This map is invertible. Its inverse is defined by $R_h^{-1}:x\mapsto x^h$.

Obviously, $R_h$ is an isomorphism with respect to multiplication, and does
not commute with the triangular addition 
$$(a,b)\mapsto a\nplus b=\{c\in\R_+\mid |a-b|\le c\le a+b\}.$$ 
In order to make $R_h$ a hyperfield isomorphism, pull back the triangular 
addition, that is define
$$
a\nplus_h b=R_h^{-1}(R_h(a)\nplus R_h(b))=
\{c\in\R_+\mid |a^{1/h}-b^{1/h}|^h\le
c\le(a^{1/h}+b^{1/h})^h\}.
$$
Observe that if $a\ne b$, then 
$$\lim_{h\to0}|a^{1/h}-b^{1/h}|^h=\lim_{h\to0}(a^{1/h}+b^{1/h})^h=\max(a,b),$$
and if $a=b$, then $|a^{1/h}-b^{1/h}|^h=0$, while
$\lim_{h\to0}(a^{1/h}+b^{1/h})^h=a$.
Thus the endpoints of the segment $a\nplus_hb$ tend to the endpoints of the
segment $a\yplus b$ as $h\to0$. Define $a\nplus_0b$ to be $a\yplus b$. 

For $h\ge0$, denote by $\mftr_h$ the hyperfield with the underlying 
set $\R_+$,
addition $\nplus_h$ and multiplication coinciding with the usual
multiplication of real numbers. For $h>0$, the map $R_h$ is an isomorphism 
$\mftr_h\to\mftr$. The hyperfield $\mftr_0$ coincides with
$\mfutr$. 

Thus, $\mftr_h$ is a dequantization (degeneration) of $\mftr$ to $\mfutr$. 
The map $\log:\R_+\to\R\cup\{-\infty\}$ converts $\mftr_h$ into a
dequantization of the amoeba hyperfield $\mfamb$ to the tropical
hyperfield $\mftrop$.

\subsection{Dequantization of $\C$ to $\tc$}\label{s9.3}
For a positive real number $h$, let $S_h{:\ }\C\to\C$ be the map defined by
the formula
$$z\mapsto 
\begin{cases}
|z|^\frac1h\frac{z}{|z|} &\text{ for }z\ne0;\\
0 &\text{ for }z=0.
\end{cases}$$
This map is invertible. Its inverse is defined by the formula
$$S_h^{-1}:z\mapsto 
\begin{cases} 
|z|^h\frac{z}{|z|} &\text{ for }z\ne0;\\
0 &\text{ for }z=0.
\end{cases}$$

Obviously, $S_h$ is an isomorphism with respect to multiplication, that is 
$S_h(ab)=S_h(a)S_h(b)$. However, it does not commute with addition.

In order to make $S_h$ an isomorphism with respect to addition, let us
redefine the addition on the source of the map. In other words, induce a
binary operation on the set of complex numbers: 
$$a+_hb=S_h^{-1}(S_h(a)+S_h(b)).$$
In this way we get a field $\C_h=(\C,+_h,\times)$ (which is nothing
but a copy of $\C$) and an isomorphism $S_h:\C_h\to\C$.

It is easy to see that $a\plus_hb$ converges as  $h$ tends to zero. 
Namely:
\begin{itemize} 
\item if $|a|>|b|$, then $\lim_{h\to0}(a\plus_hb)=a$;
\item if $|a|=|b|$ and $a+b\ne0$, then 
$\lim_{h\to0}(a\plus_hb)=|z|\frac{a+b}{|a+b|}$;
\item if $a+b=0$, then $\lim_{h\to0}(a\plus_hb)=0$.
\end{itemize} 
Denote $\lim_{h\to0}(a+_hb)$ by {$a+_0b$}. See figure \ref{t0C}. 
\begin{figure}[htb]
\centerline{$\vcenter{\hbox{\input{figs/t0C.pstex_t}}}$}
\caption{The limit $a+_0b$ of $a\plus_hb$ as $h\to0$. }
\label{t0C}
\end{figure}

Some properties of the operation $(a,b)\mapsto a+_0b$ are nice.
It is commutative, distributive for the usual multiplication of complex
numbers, the zero behaves appropriately: $a+_00=a$ for any $a\in\C$.
Furthermore, for any $a\in\C$ there exists a unique complex number $b$ such
that  $a+_0b=0$, and this $b$ is nothing but  $-a$. 

However, the operation $(a,b)\mapsto a+_0b$ is far from being perfect.
First, $a+_0b$ is not continuous as a function of  $a$ and $b$. Certainly,
this happens because the convergence $a+_hb\to a+_0b$ is not uniform. 
Second, it is not associative. 

In order to see the latter, compare
$(-1+_0i)+_01$ и $-1+_0(i+_01)$:
\begin{multline*}(-1+_0 i)+_0 1=
\left(\exp\left(\pi i\right)+_0\, 
\exp\left(\frac{\pi i}2\right)\right)+_0\, 1\\
=\exp\left(\frac{3\pi i}4\right)+_0\,
\exp(0)=\exp\left(\frac{3\pi i}8\right)
\end{multline*}
On the other hand,
\begin{multline*}
-1\plus_0(i+_0 1)=\exp(\pi i)+_0\left(\exp\left(\frac{\pi
i}2\right)+_0\, \exp(0)\right)\\
=\exp(\pi i)+_0\, \exp\left(\frac{\pi i}4\right)=\exp\left(\frac{5\pi
i}8\right).
\end{multline*}

The tropical addition $(a,b)\mapsto a\cplus b$ 
introduced in Section \ref{s6.1} above does not have
this defect. It is associative, see Appendix 1. Though, it is multivalued.

Another advantage of the tropical addition is that it is upper
semicontinuous, see Section \ref{s8.3}. The tropical addition is also a
limit of $\plus_h$ as $h\to0$, but not in the sense of pointwise 
convergence.

\begin{Th}\label{th-deq}
Let $\GG=\{(a,b,a+_hb,h)\in \C^3\times\R_+\mid h>0, a\in\C,b\in\C\}$.
Then the intersection of $\C^3\times\{0\}$ with the closure of $\GG$ is the
$\GG_{\cplus}\times\{0\}$, where $\GG_{\cplus}$ is the graph 
$\{(a,b,a\cplus b)\mid a\in\C,b\in\C\}$ of $(a,b)\mapsto a\cplus b$. 
\end{Th}

Thus the tropical addition of complex numbers is a dequantization of the
usual addition of complex numbers in the same way as taking maximum is a
dequantization of the usual addition of positive real numbers.

\begin{proof}[{\bfseries Proof of Theorem \ref{th-deq}}]
In order to prove the equality 
$$\GG_{\cplus}\times\{0\}=\Cl(\GG)\cap\C^3\times\{0\},$$
that is the content of Theorem \ref{th-deq}, we will
prove the corresponding two inclusions:
\begin{equation}\label{incl1}
\GG_{\cplus}\times\{0\}\subset\Cl(\GG)\cap\C^3\times\{0\},
\end{equation}
\begin{equation}\label{incl2}
\GG_{\cplus}\times\{0\}\supset\Cl(\GG)\cap\C^3\times\{0\}.
\end{equation}

\noindent{\bfit Proof of \eqref{incl1}.}
There are three types of points in $\GG_{\cplus}\subset\C^3$:
\begin{enumerate} 
\item $(a,b,a)$ with $|a|>|b|$, or $(a,b,b)$ with $|a|<|b|$;
\item $(a,b,c)$ with $|a|=|b|=|c|$ and $a+b\ne0$;
\item $(a,-a,b)$ with $|b|\le|a|$. 
\end{enumerate}
In the first case, if $|a|>|b|$, then $a=a+_0b$, hence $(a,b,a)$ 
belongs to the graph of $+_0$, and 
therefore $(a,b,a,0)$ belongs to the closure of $\GG$.
If $|a|<|b|$, then $b=a+_0b$, hence $(a,b,b)$ 
belongs to the graph of $+_0$, and 
therefore $(a,b,b,0)$ belongs to the closure of $\GG$.

In the second case, let $|a|=|b|=|c|=r$. Recall that $c$ belongs 
to the shortest arc connecting $a$ and $b$ on the circle $|z|=r$.
Therefore $c=\Gl a+\Gm b$ with $\Gl,\Gm\in[0,1]$. 

Assume that $c\ne a,b$. From this assumption, it follows
that $0<\Gl,\Gm<1$. Let us prove, first,  that
$c=(\Gl^ha)+_h(\Gm^hb)$. 

Indeed, 
$|S_h(\Gl^ha)|=|\Gl^ha|^{\frac1h}=\Gl |a|^{\frac1h}$ and $S_h(\Gl^ha)=
\Gl\frac{a}{|a|}|a|^{\frac1h}=\Gl ar^{\frac{1-h}h}$. Similarly, 
$S_h(\Gm^hb)=\Gm\frac{b}{|b|}|b|^{\frac1h}=\Gm br^{\frac{1-h}h}$
and $(\Gl^ha)+_h(\Gm^hb)=S_h^{-1}(S_h(\Gl^ha)+S_h(\Gm^hb))=
S_h^{-1}(r^{\frac{1-h}h}(c))=c$.

Thus $(\Gl^ha,\Gm^hb,c,h)\in\GG$. Since  $\lim_{h\to0}x^h=1$  for any 
$x\in(0,1)$,  
$$(a,b,c,0)=\lim_{h\to0}(\Gl^ha,\Gm^hb,c,h)\in\Cl(\GG).$$
Thus each interior point of the arc $(a\cplus b)\times0$ belongs to
the closure of $\GG$. Therefore, its boundary points belong to the closure
of $\GG$, too.  

Consider finally the last case, $(a,-a,b)$ with $|b|\le|a|$. It would 
suffice to prove that $(a,-a,b,0)$ belongs to the closure of $\GG$ for 
$b$ with $|b|<|a|$. Obviously, $(a+_hb,-a,b,h)$ belongs to $\GG$. Indeed,
$(a+_hb)+_h(-a)=a+_h(-a)+_hb=b$. Further, $\lim_{h\to0}(a+_hb)=a+_0b=a$,
since $|b|<|a|$.\medskip

\noindent{\bfit Proof of \eqref{incl2}.} The Inclusion \eqref{incl2}
follows from the following lemmas.

\begin{lem}\label{lem-d6}
If $(a,b,c,0)\in\Cl\GG$, then $|c|\le\max(|a|,|b|)$.
\end{lem}

\begin{lem}\label{lem-d-7}
If $(a,b,c,0)\in\Cl\GG$ with $|a|>|b|$, then
$c=a$.
\end{lem}

\begin{lem}\label{lem-d5}
If $(a,b,c,0)\in\Cl\GG$ and $|a|=|b|$, but $a+b\ne0$, then
$|c|\ge |a|$.
\end{lem}

\begin{lem}\label{lem-d8}
If $(a,b,c,0)\in\Cl\GG$ and $|a|=|b|$, but
$a+b\ne0$, then $c\in a\R_++b\R_+$.
\end{lem}

\begin{proof}[{\bfseries Proof of Lemma \ref{lem-d6}}]
\begin{multline*}
|a+_hb|=|S_h^{-1}(S_h(a)+S_h(b))|\\
=|S_h(a)+S_h(b)|^h\\
\le (|S_h(a)|+|S_h(b)|)^h\\
\le(2\max(S_h(a),S_h(b))^h\\
=2^h\left(\max(|a|^{\frac1h},|b|^{\frac1h})\right)^h\\
=2^h\max(|a|,|b|)
\end{multline*}
Since $2^h\underset{h\to0}{\to}1$, it follows that for any 
$C>\max(|a|,|b|)$ there exists neighborhoods $U$ and $V$ of $a$ and $b$,
respectively, and a real number $\Ge>0$ such that 
$\sup\{|x|\mid x\in U+_hV\}$ is not greater than $C$ for any $h\in(0,\Ge)$.
\end{proof}

\begin{proof}[{\bfseries Proof of Lemma \ref{lem-d-7}}]
For any complex numbers $x,y$ with $|x|>|y|$, 
\begin{multline*} 
x+_hy=S_h^{-1}(S_h(x)+S_h(y))=\\
S_h^{-1}\left(S_h(x)\left(1+\frac{S_h(y)}{S_h(x)}\right)\right)=
xS_h^{-1}\left(1+\frac{S_h(y)}{S_h(x)}\right)
\end{multline*}
Further, $|\frac{S_h(y)}{S_h(x)}|=\left|\frac{y}x\right|^{\frac1h}$.
Hence $\left|1+\frac{S_h(y)}{S_h(x)}\right|\le
1+\left|\frac{y}x\right|^{\frac1h}$
and 
$$\left|S_h^{-1}\left(1+\frac{S_h(y)}{S_h(x)}\right)\right|=
 \left|1+\frac{S_h(y)}{S_h(x)}\right|^h
\le\left|1+\left|\frac{y}{x}\right|^{\frac1h}\right|^h.$$
The family $|1+a^{\frac1h}|h$ converges to 1 as $h\to0$ uniformly for 
$a\in (0,r)$ if $r<1$. Therefore 
$x+_hy$ converges to $x$ as $h\to0$ uniformly on the set $|x|\le R$ and
$|y|\le r$ if $R$ and $r$ are positive real numbers with $0<r<R$.

If $(a,b,c,0)\in\Cl\GG$ and $|a|>|b|$, then for any neighborhoods $U$, $V$
and $W$ of $a$, $b$ and $c$, respectively, and any $\Ge>0$ there exist 
$h\in(0,\Ge)$ and $(x,y)\in U\times V$ such that $x+_hy\in W$.
Let $R$ and $r$ be real numbers with $|a|>R>r>|b|$.
We may take neighborhoods $U$ and $V$ such that $|y|<r$ ad $R<|x|$
for any $x\in U$ and $y\in V$. When $(x,y)\in U\times V$,  
$x+_hy$ uniformly converges to $x$ as $h\to0$. On the other hand we see
that by shrinking $W$ towards $c$ and pushing $\Ge$ to 0, we force
$x+_hy$ converge to $c$, while by shrinking $U$ towards $a$, we force 
$x$ converge to $a$. Hence $c=a$.  
\end{proof}

\begin{proof}[{\bfseries Proof of Lemma \ref{lem-d5}}]
The numbers $a$ and $b$ can be related by formula $b=ae^{i\Gf}$ with 
$|\Gf|<\pi$. Then $S_h(b)=S_h(ae^{i\Gf})=e^{i\Gf}S_h(a)$
and $|a+_hb|
=|S_h^{-1}(S_h(a)+S_h(b))|
=|S_h^{-1}(S_h(a)(1+e^{i\Gf}))|
=|a||1+e^{i\Gf}|^h$

\end{proof}

\begin{proof}[{\bfseries Proof of Lemma \ref{lem-d8}}]
Fix $h>0$. The numbers $S_h(a)$ and $S_h(b)$ have the same arguments
as $a$ and $b$. Therefore their sum $S_h(a)+S_h(b)$ belongs to
$a\R_++b\R_+$.  The number
$a+_hb=S_h^{-1}(S_h(a)+S_h(b))$ has the same argument as $S_h(a)+S_h(b)$.
Hence, it also belongs to $a\R_++b\R_+$.
\end{proof}

\end{proof}

\subsection{Dequantizations commute}\label{s9.4}
We have constructed the following three 1-parameter families of hyperfields:
\begin{itemize} 
\item $\mftr_h$ degenerating the triangle hyperfield $\mftr$ to the
ultratriangle hyperfield $\mfutr$;
\item $\mfamb_h$ degenerating the amoeba hyperfield $\mfamb$ to the
tropical hyperfield $\mftrop$;
\item $\C_h$ degenerating the field $\C$ of complex numbers to the 
complex tropical hyperfield $\tc$.
\end{itemize}
These families are related. The map $\log:\R_+\to\R\cup\{-\infty\}$
maps the first of them to the second one. This is how the second
deformation was obtained. Furthermore, the map $\C\to\R_+:z\mapsto|z|$ maps
the third deformation onto the first one. The composition of these two 
maps, the amoeba map $\C\to\R\cup\{-\infty\}:z\mapsto\log|z|$, maps 
the third deformation to the second one. 

$$
\begin{CD}\C\cong\C_h@>{h\to0}>>\C_0=\tc\\
             @V{x\mapsto|x|}VV   @VV{x\mapsto|x|}V\\
            \mftr\cong\mftr_h@>>{h\to0}>\mftr_0=\mfutr\\
            @V{x\mapsto\log x}VV   @VV{x\mapsto\log x}V\\
             \mfamb\cong\mfamb_h@>>{h\to0}>\mftrop
\end{CD}
$$

All vertical arrows in this diagram are 
multiring homomorphisms discussed above. The horizontal arrows denote 
passing to limits.

Double distributivity of a hyperfield is not preserved under
dequantization. In the first line of the diagram the original hyperfield 
is a field $\C$. It is doubly distributive. The complex
tropical hyperfield is not (see Section \ref{s6.4}). 
In the second and third lines the original
hyperfields are not doubly distributive (see Section \ref{s5.1}), 
while the dequantized hyperfields are (cf. Sections \ref{s4.5}, \ref{s5.2}
and \ref{s5.3}).  

Each of the hyperfields gives rise to its own algebraic geometry.
The classical complex algebraic geometry corresponds to the left upper
corner of the diagram. The left vertical arrows correspond to 
construction of amoeba for a complex algebraic variety. The bottom 
right corner of the diagram corresponds to the tropical geometry.

The least studied of these algebraic geometries is the one corresponding to
the right upper corner of the diagram. This is the complex tropical
geometry. It occupies an intermediate position between the complex
algebraic geometry and tropical geometry, cf. \cite{Viro_tg}.

\section*{Appendix 1. Proof of Theorem \ref{th1}}\label{sA1}
 Let us prove that 
$(a\cplus b)\cplus c=a\cplus(b\cplus c)$ 
for any complex numbers $a,b,c$. The following list exhausts all possible 
triples of complex numbers:
\begin{enumerate} 
\item the absolute value of one of the numbers, say $a$, is greater than
the absolute values of the other two numbers: $|a|>|b|,|c|$;
\item $|a|=|c|>|b|$; 
\item $|a|=|b|>|c|$ and  
\begin{enumerate} 
\item  either $a\ne-b$,
\item or $a=-b$;
\end{enumerate}
\item $|a|=|b|=|c|$ and
\begin{enumerate}
\item $a+b\ne0\ne b+c$;
\item either $a+b=0$, or $b+c=0$, but not both;
\item $a+b=0=b+c$, but $a\ne0$;
\item $a=b=c=0$.  
\end{enumerate} 
\end{enumerate} 
Let us prove that 
$(a\cplus b)\cplus c=a\cplus(b\cplus c)$ in each of these cases. 
In the framework of the prove in the case when $x\cplus y$ is an arc 
(i.e., $|x|=|y|$ and $x+y\ne0$), let us denote this arc by 
$\mathop\frown{(xy)}$. 

(1) In the first case (i.e., if $|a|>|b|,|c|$) the tropical sum equals $a$,
that is the summand with the greatest absolute value independently on the
order of operations. For any order this summand majorizes the others and
eventually becomes the final result.\qed

(2) If $|a|>|b|$ and $|b|<|c|$, then $a\cplus b=a$ and $b\cplus c=c$. 
Hence $(a\cplus b)\cplus c=a\cplus c$ and $a\cplus(b\tplus c)=a\tplus c$.

(3a) If $|a|=|b|$ and $a\ne-b$, then $a\cplus
b=\mathop\frown{(ab)}$, and since $|c|<|a|$, then $c\cplus x=x$ for any
$x$ with $|x|=|a|$.  Therefore    
$(a\cplus b)\cplus c=(\mathop\frown{(ab)})\cplus c=\mathop\frown(ab)$. 
On the other hand, $a\cplus(b\cplus c)=a\cplus b=\mathop\frown(ab)$.\qed

(3b) 
\begin{multline*}(a\cplus -a)\cplus c=
\{x\mid|x|\le|a|\}\cplus c=\\
\left(\begin{aligned}
&\{x\mid |c|<|x|\le|a|\} \cup\\
&\{x\mid |x|=|c|, x\ne-c\}\cup\\
&\{-c\}\cup\\
&\{x\mid |x|<|c|\}
\end{aligned}
\right)\cplus c=
\left(
\begin{aligned}
&\{x\mid |c|<|x|\le|a|\}\cup\\ 
&\{y\mid |y|=|c|, x\ne-c\}\cup\\
&\{x\mid |x|\le|c|\}\cup\\
&\{c\}
\end{aligned}
\right)=\\
\{x\mid |x|\le|a|\}\end{multline*}
On the other hand,
$a\cplus(-a\cplus c)=a\cplus(-a)= \{x\mid |x|\le |a|\}$
\qed

(4a)
\begin{multline*}
(a\cplus b)\cplus c=(\mathop\frown(ab))\cplus c=\\
\begin{cases}
\{x\mid|x|\le|a|\}, &\text{ if }-c\in (\mathop\frown(ab))\\
(\mathop\frown(ac))\cup(\mathop\frown(bc)), &\text{ if } 
-c\not\in (\mathop\frown(ab))
\end{cases}
\end{multline*}
On the other hand,  
\begin{multline*}
a\cplus(b\cplus c)=a\cplus(\mathop\frown(bc))=\\
\begin{cases} 
\{x\mid|x|\le|a|\}, &\text{ if } -a\in (\mathop\frown(bc))\\
(\mathop\frown(ab))\cup(\mathop\frown(ac)), &\text{ if } 
-a\not\in (\mathop\frown(bc))
\end{cases}
\end{multline*}

The statements $-c\in (\mathop\frown(ab))$ and $-a\in (\mathop\frown(bc))$
are equivalent. Indeed, each of them means that the convex hull of the set 
 $\{a,b,c\}$ contains 
$0$. If the convex hull of $\{a,b,c\}$ does not contain
$0$, then  $\{a,b,c\}$ is contained in a half of the circle 
$\{x\mid |x|=|a|\}$ and then 
$(\mathop\frown(ac))\cup(\mathop\frown(bc))=(\mathop\frown(ab))\cup(\mathop\frown(ac))$
is the shortest arc of the circle containing $a,b,c$, that is it is a sort
of convex hull of $\{a,b,c\}$ in a half-circle. 
\qed

(4b) If $|a|=|b|=|c|$, $a+b=0$, but  $b+c\ne0$, then
$(a\cplus b)\cplus c=
\{x\mid |x|\le|a|\}\cplus c=
=(\{-c\}\cup\{x\mid x\ne-c |x|\le|a|\}=\{x\mid |x|\le|a|\}$.
On the other hand,
$a\cplus(-a\cplus c)=a\cplus(\mathop\frown(-a,c)= \{x\mid|x|\le|a|\}$. \qed

(4c) If $|a|=|b|=|c|\ne0$ and $a+b=0=b+c$, then  
$(a\cplus b)\cplus c=(a\cplus-a)\cplus a=\{x\mid|x|\le|a|\}\cplus a=
\{x\mid |x|\le|a|\}$. On the other hand,
$a\cplus (b\cplus c)=a\cplus(-a\cplus a)=a\cplus\{x|\mid |x|\le|a|\}=
\{x\mid |x|\le|a|\}.$\qed

(4d) Does not require a proof. \qed

\section*{Appendix 2. Proof of Theorem \ref{th-1}}\label{sA2}

For $n=2$ the statement of Theorem \ref{th-1} follows
immediately from the definition of tropical sum. Assume that for all  
$n<k$ the statement is proved and prove it for  $n=k$. 

By the assumption, the tropical sum of the first $k-1$ summands is either
the whole closed disk, and then  
$0\in\Conv(a_1,\dots,a_{k-1})$, or $a_1\cplus\dots\cplus a_{k-1}$ is a
connected subset of a half of the circle. In the former case the sum of all
$k$ summands is the same disk, since  $-a_k\in a_1\cplus\dots\cplus
a_{k-1}$, and $0\in\Conv(a_1,\dots,a_{k})$, since 
$0\in\Conv(a_1,\dots,a_{k-1})$. 

In the latter case there may happen one of the following two mutually
exclusive situations: either  $-a_k\in a_1\cplus\dots\cplus a_{k-1}$, 
and then $a_1\cplus\dots\cplus a_k$ is the
disk, or  $-a_k\not\in a_1\cplus\dots\cplus a_{k-1}$.

In the first situation, the diameter of the disk which connects $a_k$ and 
$-a_k$ meets the chord subtending the arc $a_1\cplus\dots\cplus a_{k-1}$. 
(We do not exclude the case when $a_1\cplus\dots\cplus a_{k-1}$ is a point, 
but just consider a point as a degenerated arc. All the arguments below
have obvious versions for this case.)  The center of the
disk lies on the part of the diameter connecting $a_k$ with the chord
subtending the arc $a_1\cplus\dots\cplus a_{k-1}$. The end points of the
arc are some of the first $k-1$ summands by the induction assumption.
Therefore, $0\in\Conv(a_1,\dots,a_k)$.

In the second situation (that is if  
$-a_k\not\in a_1\cplus\dots\cplus a_{k-1}$) 
either $a_k\in a_1\cplus\dots\cplus a_{k-1}$ and then 
$a_1\cplus\dots\cplus a_{k}= a_1\cplus\dots\cplus a_{k-1}$, so the second 
alternative takes place, or 
$a_k\not\in a_1\cplus\dots\cplus a_{k-1}$ and then $a_1\cplus\dots\cplus
a_{k-1}$ lies 
on one side of the
diameter connecting $a_k$ with 
$-a_k$. In the latter case  $a_1\cplus\dots\cplus a_{k}$ is an arc one of
the end points of which is $a_k$, while the other end point is one of the
end points of the arc
$a_1\cplus\dots\cplus a_{k-1}$.\qed

\section*{Appendix 3. Other tropical additions}\label{sA3}

\subsection*{A3.1 Tropical addition of quaternions}\label{sA3.1}
Denote by $\H$ the skew field of quaternions 
$\{x+y\mathbf{i}+z\mathbf{j}+t\mathbf{k}\mid x,y,z,t\in\R\}$.
Let $a,b\in\H$. 
Like in Section \ref{s6.1} define 
$$
a\cplus b = 
\begin{cases}
\{a\}, &\text{ if }\quad |a|>|b|;\\
\{b\}, &\text{ if }\quad |a|<|b|;\\
\text{\parbox{2.3in}{the set of points on the shortest geodesic arc
connecting  $a$ and $b$ in the sphere $\{c\in\H\mid |c|=|a|\}$,}} 
&\text{ if }
|a|=|b|, \  a+b\ne0\\  
\text{the ball }\{c\in\H \mid|c|\le|a|\}, &\text{ if }\quad a+b=0.   
\end{cases} 
$$
Let us call the set $a\tplus b$ the {\sfit tropical sum\/}
of quaternions $a$ and $b$. 

\begin{Th}\label{th11}
The set $\H$ equipped with the tropical addition is a commutative
multigroup.
\end{Th} 

The proof reproduces almost literally the proof of Theorem \ref{th1}.\qed

It is easy to verify that the quaternion multiplication is distributive
over the tropical addition. Thus we have a {\sfit skew hyperfield.}

\subsection*{A3.2 Vector spaces over $\tc$}\label{sA3.2}
The construction of tropical addition of quaternions is a special case of a
more general construction. In an arbitrary normed vector space  $V$ over 
$\C$, define multivalued operation $(a,b)\mapsto
a\tplus b$:
 $$
a\cplus b = 
\begin{cases}
\{a\}, &\text{ if }\quad |a|>|b|;\\
\{b\}, &\text{ if }\quad |a|<|b|;\\
\Cl\left\{
\dfrac{|a|}{|\Ga a+\Gb b|}(\Ga a+\Gb b)\in V\mid \Ga,\Gb\in\R_{>0}\right\}, 
&\text{ if }
|a|=|b|, \ a+b\ne0\\  
\{c\in V \mid|c|\le|a|\}, &\text{ if }\quad a+b=0.   
\end{cases} 
$$

This operation turns $V$ into a multigroup and satisfies two
kinds of distributivity: $a(v\cplus w)=av\cplus aw$ and 
$av\tplus bv=(a\tplus b)v$ where $a,b\in \C$ and $v,w\in V$. In other
words, $V$ becomes a vector space over  $\tc$ in the sense of the following
definition.

Let $F$ be a hyperfield. A set  $V$ with a multivalued binary operation 
$(v,w)\mapsto v\tplus w$ and with an action 
$(a,v)\mapsto av$, $a\in F$, $v\in V$ 
 of the multiplicative group 
of $F$ is called a {\sfit vector space\/} over $F$ if 
\begin{itemize} 
\item $\tplus$ defines in $V$ a structure of commutative multigroup; 
\item $(ab)v=a(bv)$  for any $a,b\in F$ and $v\in V$;
\item $1v=v$  for any $v\in V$;
\item $a(v\tplus w)=av\tplus aw$ for any $a\in F$ and $v,w\in V$;
\item $(a\tplus b)v=av\tplus bv$ for any $a,b\in F$ and $v\in V$.
\end{itemize} 

Of course, any hyperfield is a vector space over itself. Copies of this
vector space are contained in any vector space over a hyperfield. Indeed,
if $V$ is a vector space over a hyperfield $F$ and $w\in V$, 
then the subset $W=\{aw\mid a\in F\}$ is a vector subspace of $V$ in the
obvious sense, the map  $F\to V:a\mapsto aw$ maps 
$F$ onto $W$ and this map is an isomorphism of vector spaces. 

As in a category of vector spaces over a field, the Cartesian product 
$V\times W$ of vector spaces $V$, $W$ over a hyperfield $F$ is naturally
equipped with structure of vector space over  $F$:
\begin{align*}
(v_1,w_1)\tplus(v_2,w_2)&=\{(v,w)\mid v\in v_1\tplus v_2, \ w\in w_1\tplus
w_2\}\\
a(v,w)&=(av,aw).
\end{align*}

Notice, however that, in contrast to vector spaces over a field, a vector
space over a hyperfield generated by a finite set of its elements is not 
necessarily isomorphic to the Cartesian product of its vector subspaces
each of which is generated by a single element. Indeed, a vector space over 
$\tc$ constructed in the way described above starting from a two-dimensional 
Hilbert space over  $\C$, is not isomorphic to $\tc\times \tc$.   

\subsection*{A3.3 Hyperfields of monomials}\label{sA3.3}
The next example was inspired by  Brett Parker's paper \cite{Parker}, which 
was also motivated by a desire to understand tropical degenerations of
complex structures. 

What if one would apply the construction of Section \ref{s7.8}, 
but taking into account the absolute value of the coefficient in the
monomial of the highest degree? 

Consider the set of monomials  $at^r$ with complex coefficient 
$a\ne0$ and real exponent $r$. Adjoin zero to this set. As a set, this is
 $(\C\sminus0)\times\R\cup\{0\}$. Denote it by $P$ and define in it
arithmetic operations. 

Define multiplication as the usual multiplication of monomials. The set of
non-zero monomials is an abelian group with respect to the multiplication.
This group is naturally isomorphic to the product of the multiplicative
group of non-zero complex numbers by the additive group of all real
numbers.

Define multivalued addition by the following formulas:
$$
\begin{aligned}
&at^r\tplus bt^s=
\begin{cases} 
at^r, &\text{ if } r>s\\
bt^s, &\text{ if } s>r\\
(a+b)t^r, &\text{ if } s=r, a+b\ne0\\
\{ct^u\mid u<r\}\cup\{0\} &\text{ if } s=r, a+b=0,
\end{cases}
\\
&0\tplus x=x.
\end{aligned}
$$
This addition is obviously commutative. The multiplication is distributive 
over it. 
There is neutral element $0$ and for each monomial $x$ there is a unique
$y$ such that $x\tplus y\ni0$. Let us verify associativity. 

If one of the summands is zero, then associativity takes place and the
proof is obvious: 
$(x\tplus0)\tplus y=x\tplus y=x\tplus(0\tplus y)$.

Consider three non-zero monomials, $at^u$, $bt^v$ and $ct^w$. 
The following list represent all possibilities:
\begin{enumerate}
\item the exponents of one of the monomials is greater than the exponents
of the other two monomials,
say, $u>v,w$;
\item  two exponents, say $u$ and $v$, are equal, while the third one is
less, and $a+b\ne0$;
\item  two exponents, say $u$ and $v$, are equal, the third is less, and
$a+b=0$;
\item all the three exponents are equal and either 
\begin{enumerate}
\item  none of the sums $a+b$, $b+c$, $a+b+c$ vanishes;
\item  or the sum of two coefficients vanishes, say $a+b=0$, 
(but $a+b+c\ne0$);
\item or $a+b+c=0$.
\end{enumerate}
\end{enumerate}

Let us prove associativity in each of these cases.

(1) The sum equals the summand with the greatest exponent independently on
the order of operations. For any order this summand is the final result. 
\qed

(2) 
$(at^u\tplus bt^u)\tplus ct^w=(a+b)t^u\tplus ct^w=(a+b)t^u$, on the other
hand, $at^u\tplus(bt^u\tplus ct^w)=at^u\tplus bt^u=(a+b)t^u$.\qed

(3) 
\begin{multline*}(at^u\tplus -at^u)\tplus ct^w=
(\{xt^r\mid r<u\}\cup\{0\})\tplus ct^w=\\
\left(\begin{aligned}
&\{xt^r\mid w<r<u\} \cup\\
&\{xt^r\mid r=w, x\ne-c\}\cup\\
&\{-ct^w\}\cup\\
&\{xt^r\mid r<w\}\cup\{0\}
\end{aligned}
\right)\tplus ct^w=
\left(
\begin{aligned}
&\{xt^r\mid w<r<u\}\cup\\ 
&\{yt^w\mid y\ne0,y\ne c\}\cup\\
&\{xt^r\mid r<w\}\cup\{0\}\cup\\
&\{ct^w\}
\end{aligned}
\right)=\\
\{xt^r\mid r<u\}\cup\{0\}\end{multline*}
On the other hand,
$$
at^u\tplus(-at^u\tplus ct^w)=at^u\tplus(-at^u)= \{xt^r\mid r<u\}\cup\{0\}
$$\qed

(4a)
$(at^u\tplus bt^u)\tplus ct^u=(a+b)t^u\tplus ct^u=(a+b+c)t^u$ and 
$at^u\tplus(bt^u\tplus ct^u)=at^u\tplus(b+c)t^u=(a+b+c)t^u$.\qed

(4b) If $a+b=0$, and none of the sums  $b+c$, $a+b+c$ vanishes, then
$(at^u\tplus -at^u)\tplus ct^u=
(\{xt^r\mid r<u\}\cup\{0\})\tplus ct^u=
=ct^u$
On the other hand,
$at^u\tplus(-at^u\tplus ct^u)=at^u\tplus(-a+c)t^u= ct^u$. \qed

(4c) If all three exponents equal and $a+b+c=0$, then  
$$(at^u\tplus bt^u)\tplus ct^u=(a+b)t^u\tplus ct^u=(-c)t^u\tplus ct^u=
\{xt^r\mid r<u\}\cup\{0\}$$, 
on the other hand,
$$at^u\tplus (bt^u\tplus ct^u)=at^u\tplus(b+c)t^u=at^u\tplus(-a)t^u=
\{xt^r\mid r<u\}\cup\{0\}.$$\qed

\noindent{\bf Remark. } There are numerous variants of this construction. 
For example, in the definition of the addition of monomials all the
inequalities can be reverted. Another opportunity for modification: 
restrict consideration
to monomials whose exponents take only rational or integer values. More
generally, exponents can be taken from any linearly ordered abelian group.

\subsection*{A3.4 Tropical addition of $p$-adic numbers.}\label{sA3.4} 
Construction of Section \ref{s7.8} admits a modification 
applicable to any field with a non-archimedian norm.
In any such field one can define a multivalued addition which together with
the original multiplication form a structure of hyperfield. Below this
scheme is realized only in the case of field of $p$-adic numbers. The
general case will be considered elsewhere. 

Recall that a $p$-adic number can be defined as series 
$$
\sum_{n=-v(a)}^{\infty}a_np^n,
$$
where  $a_n$ takes values in the set of integers from the interval 
$[0,p-1]$ and $a_{-v(a)}\ne0$. Define a multivalued sum of $p$-adic numbers 
$a=\sum_{n=-v(a)}^{\infty}a_np^n$ and $b=\sum_{n=-v(b)}^{\infty}b_np^n$ 
via the following formula:
\begin{equation}\label{eqP-ad}
a\tplus b=
\begin{cases} 
a, &\text{ if } v(a)>v(b);\\
b, &\text{ if } v(b)>v(a);\\
a+b, &\text{ if } v(a)=v(b), \ a_{-v(a)}+b_{-v(b)}\ne p;\\
\{x\mid v(x)<v(a)\}, &\text{ if } v(a)=v(b), \ a_{-v(a)}+b_{-v(b)}=p.
\end{cases}
\end{equation}
Exactly as in the preceding Section, one can prove that this binary
multivalued operation is associative and, together with the usual
multiplication, gives rise to a structure of multivalued field in the set
of $p$-adic numbers.


\begin{thebibliography}{99}

\bibitem{BR}  V.~M.~Buchstaber and E.~G.~Rees, {\em Multivalued groups,
their transformations and Hopf algebras,\/} Transform. Groups 2 (1997),
325-349.

\bibitem{Comer} S.~D.~Comer, {\em Combinatorial aspects of relations\/},
Algebra Universalis, 18 (1984) 77-94.

\bibitem{CC1} Alain Connes and Caterina Consani, {\em The hyperring of 
adele classes,\/}  arXiv: 1001.4260 [mathAG,NT].

\bibitem{CC2} Alain Connes and Caterina Consani, {\em From monoids to 
hyperstructures: in search of an absolute arithmetic.\/}  
arXiv:1006.4810v1 [math.AG].

\bibitem{DO} M.~Dresher, O.~Ore, {\em Theory of Multigroups,\/}
Amer.J.Math. 60 (1938), 705-733. 

\bibitem{Krasner1} Marc Krasner,  Approximation des corps valu\'es 
complets de caract\'eristique $p\ne0$ par ceux de
caract\'eristique 0, (French) 1957 Colloque d'alg\`ebre sup\'erieure, 
tenu \`a Bruxelles du 19 au 22 d\'ecembre 1956 pp. 129-206 
Centre Belge de Recherches Math\'ematiques \'Etablissements
Ceuterick, Louvain; Librairie Gauthier-Villars, Paris.

\bibitem{Krasner2} M. Krasner, {\em A class of hyperrings and
hyperfields,\/} Internat. J. Math. Math. Sci. {\bf 6} (1983),
no. 2, 307-311.

\bibitem{L-M}  G.~L.~Litvinov and V.~P.~Maslov, {\em Correspondence 
principle for idempotent calculus and some computer applications,\/}
(IHES/M/95/33), Institut des Hautes Etudes Scientifiques, 
Bures-sur-Yvette, 1995. Also in book Idempotency, J. Gunawardena (Editor), 
Cambridge University Press, Cambridge, 1998, p.420-443 and 
 arXiv:math.GM/0101021.


\bibitem{L-M-S} G.~L.~Litvinov, V.~P.~Maslov, A.~N.~Sobolevski\u i,
{\em Idempotent Mathematics and Interval Analisys,\/} Preprint math.SC/9911126,
(1999).


\bibitem{Marshall}  M.~Marshall, {\em Real reduced multirings and 
multifields\/} J. Pure and Applied
 Algebra 205 (2006), 452-468.

\bibitem{Marshall2} M.~Marshall, {\em Review of a book Valuations, orderings and Milnor K-theory, by Ido Efrat, Mathematical Surveys
  and Monographs, vol. 124, American Mathematical Society, Providence, RI,
  2006, xiv+288 pp., ISBN 978-0-8218-4041-2} 
Bulletin of the American Mathematical Society 45:3 (2008) 439-444.

\bibitem{Marty}  F.~Marty, {\em Sur une g\'en\'eralisatioii de la notioni 
de groupe,\/} S\"artryck ur F\"orhandlingar vid \"Attonde Skandinaviska 
Matematikerkongressen i Stockholm (1934),
p. 45-49.

\bibitem{Mikha} G. Mikhalkin, {\em Enumerative tropical algebraic
geometry in $\R^2$\/} , J. Amer. Math. Soc. 18 (2005), no.
2, 313-377, см. также arXiv: math.AG/0312530.

\bibitem{Mikha-ICM}  G. Mikhalkin, {\em Tropical geometry and its
applications,}\/ International Congress of Mathematicians. Vol. II, 827-852, 
Eur. Math. Soc., Z\"urich, 2006.


\bibitem{Parker} Brett Parker, {\em Exploded fibrations, \/} 
Proceedings of 13th G\"okova, Geometry-Topology Conference
pp. 1-39, arXiv:0705.2408 [math.SG]. 


\bibitem{Sturmfels}  B. Sturmfels, {\em Solving systems of polynomial 
equations,}\/ CBMS Regional Con\-ference Series in Mathematics, 
AMS Providence, RI 2002 (Chapter 9).

\bibitem{Vietoris} L. Vietoris, {\em Bereiche zweiter Ordnung\/}, 
Monatsh. f. Math. 32 (1922), 258-280.

\bibitem{Viro_msri} Oleg Viro, {\em Complex Tropical Geometry\/}, Lecture
in the workshop {\em Tropical Structures in Geometry and Physics\/}
at MSRI, November 30, 2009, 
http://198.129.64.244/13933//13933-13933-Quicktime.mov

\bibitem{Viro_bntg} Oleg Viro, {\em On basic notions of the tropical
geometry\/}, to appear in Trudy MIAN (Russian).  

\bibitem{Viro_mf} Oleg Viro, {\em Multifields for Tropical Geometry I.
Multifields and dequantization\/} arXiv:1006.3034v1.

\bibitem{Viro_eqmfd} Oleg Viro, {\em Hyperfields for Tropical Geometry II. 
Equations in a hyperfield\/},
in preparation. 

\bibitem{Viro_tg} Oleg Viro, {\em Hyperfields for Tropical Geometry III. 
Three tropical geometries\/}, in
preparation.

\bibitem{Wall1} H.~S.~Wall, {\em Hypergroups,\/} Bulletin of the American 
Mathematical Society,
vol. 41 (1935), p. 36. [Presented at the annual meeting of the A
merican Mathematical Society, Pittsburgh, December 27-31, 1934.]

\bibitem{Wall2} H.~S.~Wall, {\em Hypergroups,\/} American Journal 
of Mathematics, vol. 59 (1937) 705-733.
 
\end{thebibliography}
\end{document}